\documentclass[12pt,a4paper,reqno]{amsart}
\usepackage[all]{xy} 
\usepackage{amsmath}
\usepackage{amsfonts}
\usepackage{amssymb}
\usepackage{mathrsfs}
\usepackage{amsthm}

\usepackage{tensor}
\usepackage{cancel}
\usepackage{tikz}
\usepackage{xcolor}
\usepackage{graphicx}
\newcommand{\Bmu}{\mbox{$\raisebox{-0.59ex}
  {$l$}\hspace{-0.18em}\mu\hspace{-0.88em}\raisebox{-0.98ex}{\scalebox{2}
  {$\color{white}.$}}\hspace{-0.416em}\raisebox{+0.88ex}
  {$\color{white}.$}\hspace{0.46em}$}{}}

\usepackage{amsmath,calligra,mathrsfs}

\newcommand{\scHom}{\mathscr{H}\text{\kern -3pt {\calligra\large om}}}
\newcommand{\scExt}{\mathscr{E}\text{\kern -3pt {\calligra\large xt}}}

\DeclareMathAlphabet{\mathbbold}{U}{bbold}{m}{n}

\DeclareFontEncoding{OT2}{}{} 

\theoremstyle{plain}
\newtheorem{theorem}{Theorem}[subsection]
\newtheorem{mytheorem}{Theorem}

\newtheorem{remark}[theorem]{{\textrm{Remark}}}
\newtheorem{lemma}[theorem]{Lemma}
\newtheorem{corollary}[theorem]{Corollary}

\newtheorem*{hypothesis}{Hypothesis}

\def\le{\kern 0.03em}

\def\F{{\mathbb F}}
\def\tF{{\mathbf F}}

\def\K{{\mathcal K}}

\def\M{{\mathcal M}}

\def\O{{\mathcal O}}

\def\Q{{\mathbb Q}}

\def\Z{{\mathbb Z}}

\newcommand{\rig}{\mathrm{rig}}
\def\e{\kern 0.08em}
\def\be{\kern -.1em}
\def\lbe{\kern -.025em}

\DeclareMathOperator{\Aut}{Aut}

\DeclareMathOperator{\coh}{H}
\DeclareMathOperator{\Frob}{Frob}
\DeclareMathOperator{\Gal}{Gal}

\DeclareMathOperator{\CH}{CH} \DeclareMathOperator{\Nm}{N}
 
\DeclareMathOperator{\Sel}{{\rm{Sel}}}

\DeclareMathOperator{\Hom}{Hom}

\DeclareMathOperator{\ord}{ord} 
\DeclareMathOperator{\coker}{coker}

\begin{document}
\title[Iwasawa main conjecture for ordinary semistable elliptic curves]{Iwasawa Main Conjecture for ordinary semistable elliptic curves over global function fields}
\begin{abstract}
Let $A$ be an ordinary elliptic curve over a global function field $K$ of characteristic $p$, assumed semistable at every place, and let $L/K$ be a $\mathbb{Z}_p^d$-extension ramified only at finitely many places where $A$ has ordinary reduction. Building on the framework of \cite{tan24}, we prove the Iwasawa Main Conjecture for $A$ over $L$, subject to a technical $\mu$-invariant hypothesis that is already detected after specialization to the unramified $\mathbb{Z}_p$-extension.
The principal new input is a `$\chi$-formula' that compares appropriate $\chi$-isotypic characteristic ideals of Selmer modules with the corresponding specializations of the $p$-adic $L$-function.
Finally, to show that our $\mu$-hypothesis is non-vacuous, we prove, for $p>3$, that the hypothesis holds on a Zariski open dense locus in the moduli of semistable elliptic curves.
\end{abstract}
\author{Ki-Seng Tan}
\address{Department of Mathematics\\
National Taiwan University\\
Taipei 10764, Taiwan}
\email{tan@math.ntu.edu.tw}
\author{Fabien Trihan}
\address{Sophia University,
Department of Information and Communication Sciences
7-1 Kioicho, Chiyoda-ku, Tokyo 102-8554, JAPAN}
\email{f-trihan-52m@sophia.ac.jp}

\author{Kwok-Wing Tsoi}
\address{Department of Mathematics\\
National Taiwan University\\
Taipei 10764, Taiwan}
\email{kwokwingtsoi@ntu.edu.tw}
\maketitle
\section{Introduction}\label{s:int}

A guiding principle of Iwasawa theory is that suitably constructed $p$-adic $L$-functions should govern, in a precise ideal-theoretic sense, the structure of the corresponding Selmer modules over $p$-adic Lie extensions. In the function field setting this philosophy has become increasingly concrete, and now admits
formulations that run closely parallel to the classical picture over number fields.

In \cite{tan24}, the first-named author associated to an ordinary elliptic curve $A$ over a global function field $K$ of characteristic $p$ and to a $\Z_p^d$-extension $L/K$ a $p$-adic $L$-function,
characterised by an interpolation property for twisted Hasse--Weil $L$-values.
Beyond interpolation, \emph{loc.\,cit.} establishes the expected functional equation, proves a suite of specialization formulae, and verifies the Iwasawa Main Conjecture in several important cases. The purpose of the present article is to push this circle of ideas beyond those situations.

To be more precise, we assume that $A$ is semistable at every place of $K$ and that $L/K$ is ramified at only finitely many places of ordinary reduction for $A$ (good ordinary or multiplicative). Under a technical hypothesis on $\mu$-invariants which is formulated so as to be already \emph{detectable} after specialization to the unramified $\Z_p$-extension $K_\infty^{(p)}/K$, we prove the predicted equality between the principal ideal generated by the $p$-adic $L$-function and the characteristic ideal of the Pontryagin dual of $\Sel_{p^\infty}(A/L)$. The key new ingredient is a  `$\chi$-formula' comparing the relevant
$\chi$-isotypic characteristic ideals with the corresponding specializations of the $p$-adic $L$-function; when combined with the functional equation and with the standard specialization and restriction compatibilities on both the analytic and algebraic sides, this comparison yields the main conjecture for general $\Z_p^d$-extensions.

Finally, to underline that our $\mu$-hypothesis is far from vacuous, we show that for $p>3$ it holds generically: more precisely, we prove that it is valid on a Zariski open dense locus in the appropriate moduli space of semistable curves.

\subsection{The main conjecture} We begin by recalling the formulation of the main conjecture in the present setting.
Put $\Gamma:=\Gal(L/K)$ and $\Lambda_\Gamma:=\Z_p[[\Gamma]]$, and let
$X_L$ denote the Pontryagin dual of $\Sel_{p^\infty}(A/L)$. In the present setting one knows that $X_L$ is finitely generated over $\Lambda_\Gamma$
(see \cite[Corollary~3.2.3]{ttt24}), and hence its characteristic ideal $\CH_{\Lambda_\Gamma}(X_L)\subseteq \Lambda_\Gamma$ is well-defined under the convention that $\CH_{\Lambda_\Gamma}(X_L)=(0)$ if (and only if) $X_L$ is non-torsion.
We also write $\mu(X_L)$ for the $\mu$-invariant of $X_L$, such that $\mu(X_L)=\infty$ for non-torsion $X_L$.

On the analytic side, the main conjecture predicts the existence of a $p$-adic $L$-function. More precisely,  \cite[Definition~3.4.1]{tan24} defines an element
\[
  \mathscr L_{A/L} \in \Q_p\Lambda_\Gamma := \Q_p\otimes_{\Z_p}\Lambda_\Gamma,
\]
(denoted $\mathscr L_{A/L/K}$ when it is necessary to specify the base field $K$),
characterised by a $p$-adic interpolation property for suitably twisted Hasse--Weil
$L$-values. To keep track of $p$-divisibility of elements in $\Q_p\Lambda_\Gamma$, we employ the $\mu$-invariant which is defined as follows. For a non-zero $f\in \Q_p\Lambda_\Gamma$, we write $\mu(f)\in\Z$ for the unique integer such that
\[
  f_0 := p^{-\mu(f)}\,f \in \Lambda_\Gamma
  \qquad\text{and}\qquad
  f_0\not\equiv 0 \pmod{p},
\]
while if $f=0$, put $\mu(f)=\infty$.
A basic compatibility result in \emph{loc.\,cit.} (cf.\ Proposition~1.1.2) asserts that
\begin{equation}\label{e:mu}
  \mu(\mathscr L_{A/L}) \;=\; \mu(X_L).
\end{equation}
In particular, \eqref{e:mu} implies that $\mathscr L_{A/L}$ in fact belongs to $\Lambda_\Gamma$.

The element $\mathscr L_{A/L}$ is characterised by an interpolation property: it turns out that there exists a non-zero element $\dag_{A/L}\in\Lambda_\Gamma$, to be described explicitly in
\S\ref{su:not}, such that for every finite-order character $\omega\in\hat\Gamma$ with
$\omega(\dag_{A/L})\neq 0$ one has (\emph{loc.\ cit.}, Lemma~4.1.1)
\begin{equation}\label{e:intp}
  \omega(\mathscr L_{A/L})
  \;=\;
  \star_\omega\cdot L_{A/K}(\omega,1),
\end{equation}
where $\star_\omega$ is an explicit correction factor and $L_{A/K}(\omega,s)$ denotes the
classical $\omega$-twisted $L$-function (see \S\ref{su:not}).
Here $\hat\Gamma$ denotes the group of continuous homomorphisms
\[
  \omega:\Gamma \longrightarrow \Bmu_{p^\infty},
\]
where $\Bmu_{p^\infty}$ is given the discrete topology, so that every $\omega$ has finite order.
We extend $\omega$ uniquely to a $\Z_p$-algebra homomorphism
\[
  \omega:\Lambda_\Gamma\longrightarrow \O_\omega:=\Z_p[\omega(\Gamma)].
\]

A further point is that the interpolation property
\eqref{e:intp} determines $\mathscr L_{A/L}$ uniquely. Indeed, in the topology of Monsky, for
any non-zero element $f\in \Lambda_\Gamma$ the vanishing locus
\[
  \{\omega\in\hat\Gamma \mid \omega(f)=0\}
\]
is a proper closed subset of $\hat\Gamma$ \cite[Theorems~2.2 and~2.6]{monsky}.
Consequently, since $\dag_{A/L}\neq 0$, the values of $\omega(\mathscr L_{A/L})$ for all
$\omega$ with $\omega(\dag_{A/L})\neq 0$ uniquely specify $\mathscr L_{A/L}$.

The Iwasawa Main Conjecture for $A$ over $L/K$ is the
following equality of ideals in $\Lambda_\Gamma$:
\begin{equation}\label{e:imc}
  (\mathscr L_{A/L}) \;=\; \CH_{\Lambda_\Gamma}(X_L).
\end{equation}
It is worth noting that if $X_L$ is non-torsion then
\eqref{e:mu} already implies \eqref{e:imc}. More precisely, 
\begin{equation}\label{e:nontor}
\mathrm{CH}_{\Lambda_\Gamma}(X_L)=(0)\; \Longleftrightarrow \;\mathscr L_{A/L}=0.
\end{equation}

A useful feature of the situation at hand is that the two ideals appearing in \eqref{e:imc}
exhibit a striking parallelism. Before stating our main results, we therefore recall a small
collection of standard properties that are shared by the analytic ideal $(\mathscr L_{A/L})$ and
the algebraic ideal $\CH_{\Lambda_\Gamma}(X_L)$.
\subsection{Properties (I): Functional equations}\label{su:FE}

A first point is that both sides of the main conjecture
\eqref{e:imc} are compatible with the natural involution on the Iwasawa algebra induced by
inversion in $\Gamma$.  This compatibility, conceptually, may be viewed as the Iwasawa-theoretic shadow of the usual functional equation for complex $L$-functions, and it is one of the reasons that the
analytic and algebraic objects in \eqref{e:imc} can reasonably be expected to coincide. To make this precise, let
\[
  {}^\sharp:\Lambda_\Gamma \longrightarrow \Lambda_\Gamma,\qquad x\longmapsto x^\sharp,
\]
denote the $\Z_p$-algebra isomorphism induced by $\gamma\mapsto \gamma^{-1}$ for all
$\gamma\in\Gamma$.
Then the algebraic side satisfies the functional equation
(\cite[Theorem~2]{lltt18})
\begin{equation}\label{e:agfe}
  \CH_{\Lambda_\Gamma}(X_L) \;=\; \CH_{\Lambda_\Gamma}(X_L)^\sharp,
\end{equation}
whilst the analytic side satisfies the corresponding identity
(\cite[Proposition~4.2.3]{tan24})
\begin{equation}\label{e:anfe}
  (\mathscr L_{A/L}) \;=\; (\mathscr L_{A/L})^\sharp.
\end{equation}

\subsection{Properties (II): Specialization formulae}\label{su:SP}

A second feature of the main conjecture is its functoriality with respect to passage to intermediate extensions. Roughly speaking, the analytic and algebraic
objects in \eqref{e:imc} behave well under specialization from $\Lambda_\Gamma$ to the Iwasawa algebra of any intermediate $\Z_p^e$-extension.  This compatibility is one of the main tools in reducing multivariable statements to statements in fewer variables.

Let $L'/K$ be an intermediate $\Z_p^e$-extension of $L/K$ (with $e\geq 0$), and put
$\Gamma':=\Gal(L'/K)$.  Let
\[
  p^L_{L'}:\Lambda_\Gamma\longrightarrow \Lambda_{\Gamma'}
\]
be the $\Z_p$-algebra homomorphism induced by the natural restriction map
$\Gamma\to\Gamma'$.
Then the specialization formulae of \cite[Propositions~2.1.5 and~4.3.1]{tan24} assert that there
exist explicit principal ideals $\varrho_{L/L'}$ and $\vartheta_{L/L'}$ in $\Lambda_{\Gamma'}$
(to be described precisely in \S\ref{su:not}) such that
\begin{equation}\label{arithspf}
  \varrho_{L/L'}\cdot p^L_{L'}\!\bigl(\CH_{\Lambda_\Gamma}(X_L)\bigr)
  \;=\;
  \vartheta_{L/L'}\cdot \CH_{\Lambda_{\Gamma'}}(X_{L'}),
\end{equation}
and, provided that $p^L_{L'}(\dag_{A/L})\neq 0$, one also has the corresponding analytic
specialization
\begin{equation}\label{anspf}
  \varrho_{L/L'}\cdot p^L_{L'}\!\bigl((\mathscr L_{A/L})\bigr)
  \;=\;
  \vartheta_{L/L'}\cdot (\mathscr L_{A/L'}).
\end{equation}
In particular, up to the explicit `correction factors' $\varrho_{L/L'}$ and $\vartheta_{L/L'}$,
the two sides of \eqref{e:imc} transform under $p^L_{L'}$ in exactly the same way.
\subsection{Properties (III): Restriction formulae}\label{ss:restriction}

A third compatibility concerns change of the base field inside the extension $L/K$.
More precisely, if one replaces $K$ by a finite intermediate field $K'\subset L$, then both the analytic and algebraic sides of \eqref{e:imc} admit natural `norm-type' descent operations, and these again behave in parallel.

To be more precise, let $K'/K$ be a finite intermediate extension of $L/K$, and put $\Phi:=\Gal(L/K')$.
Choose a set of representatives $C\subset \Gamma$ for the cosets of $\Gamma/\Phi$, so that
$\Gamma=\bigsqcup_{\sigma\in C}\Phi\cdot \sigma$.  Then every $f\in\Lambda_\Gamma$ can be uniquely written as
\[
  f=\sum_{\sigma\in C} f_\sigma\cdot \sigma,
  \qquad f_\sigma\in\Lambda_\Phi.
\]
For each character $\chi\in\widehat{\Gamma/\Phi}\subset\hat\Gamma$, define $ f_\chi:=\sum_{\sigma\in C} f_\sigma\cdot \chi(\sigma)\cdot \sigma.$
Note that $f_\chi$ is independent of the choice of $C$, and the assignment
$f\mapsto f_\chi$ induces an $\O_\chi$-algebra automorphism of $\O_\chi\Lambda_\Gamma$. Then we set
\[
  f^\Gamma_\Phi \;:=\; \prod_{\chi\in\widehat{\Gamma/\Phi}} f_\chi.
\]
A key point is that $f^\Gamma_\Phi$ actually lies in $\Lambda_\Phi$ (see \cite[\S 5.1]{tan24}).
Moreover, the map $\Lambda_\Gamma\to\Lambda_\Phi$, $f\mapsto f^\Gamma_\Phi$, is multiplicative
(and satisfies $1^\Gamma_\Phi=1$).  In particular, if $I=(f)$ is a principal ideal of
$\Lambda_\Gamma$, we may define a corresponding principal ideal of $\Lambda_\Phi$ by
\[
  I^\Gamma_\Phi \;:=\; (f^\Gamma_\Phi).
\]

For a $\Lambda_\Gamma$-module $M$, we write ${\rm Res}^{\Lambda_\Gamma}_{\Lambda_\Phi}M$ for the
same underlying module regarded as a $\Lambda_\Phi$-module.
The restriction formulae of \cite[Proposition~5.1.2]{tan24} then assert that
\begin{equation}\label{e:rest}
\mathrm{CH}_{\Lambda_\Phi}\!\left(\mathrm{Res}^{\Lambda_\Gamma}_{\Lambda_\Phi} X_L\right) 
    \;=\; \mathrm{CH}_{\Lambda_\Gamma}(X_L)^\Gamma_\Phi,
\end{equation}
and that the analytic side satisfies the corresponding identity
\begin{equation}\label{e:anrest}
  (\mathscr L_{A/L/K'}) \;=\; (\mathscr L_{A/L/K})^\Gamma_\Phi.
\end{equation}

Thus, under finite base change $K'/K$, both sides of the main conjecture descend by the same
norm-type operation. 

\subsection{The $\chi$-formula}\label{su:maint}

For the remainder of the article (until \S\ref{s:isotrivial}) we impose the standing hypothesis
that $A(L)[p^\infty]$ is finite; this is automatic whenever $A/K$ is non-isotrivial (see
\cite{blv09}).
The isotrivial case will be treated separately in \S\ref{s:isotrivial}.

In Theorem \ref{t:chi} and Theorem \ref{t:1st}  below,
we shall present our $\chi$-formula by two different but equivalent statements.
Denote $L_0:=K_\infty^{(p)}$, the unramified $\Z_p$-extension of $K$.
 For the rest of this subsection (\S\ref{su:maint}),
assume that $L= L_0L_1$ where $L_0/K$ and $L_1/K$ are disjoint $\Z_p$-extensions. Write 
$\Gamma_i:=\Gal(L_i/K)$, for $i=0,1$.

The `unramified direction' $\Gamma_0$ provides a convenient axis along which one may compare the two sides of \eqref{e:imc} after twisting by finite-order characters of the `ramified direction' $\Gamma_1$. To see this, fix  a character $\chi\in\hat\Gamma_1\subset\hat\Gamma$, and let $\O_\chi$ be the $\Z_p$-order
generated by $\chi(\Gamma_1)$, with fraction field $\Q_\chi$.
Suppose $\chi$ factors through $G:=\Gal(K'/K)$ for some finite intermediate extension $K'/K$ of $L_1/K$. Put $L'_0:=L_0K'\subset L$ and denote $\Gamma_0':=\Gal(L_0'/K)$, $\Lambda_{\Gamma_0'}:=\Z_p[[\Gamma_0']]$. 
We identify $\Gamma_0'$ with $\Gamma_0\times G$, by $\gamma\mapsto (\gamma\mid_{L_0},\gamma\mid_{K'})$.
For a $ \Lambda_{\Gamma_0'}$-module $M$,
define its \(\chi\)-isotypic quotient
\[
M^{(\chi)} := M \otimes_{\Z_p[G]} \Q_\chi,
\]
where \( \Q_\chi \) is viewed as a left $G$-module via $\chi$. View $M^{(\chi)}$ as a $\Q_\chi\Lambda_{\Gamma_0}$-module.

Since $X_{L_0'}$ is finitely generated over $\Lambda_{\Gal(L_0'/K')}$, it is also finitely generated over $\Lambda_{\Gamma_0'}$, and hence $X_{L_0'}^{(\chi)}$ is finitely generated over $\Q_\chi\Lambda_{\Gamma_0}$ (see Lemma \ref{LemFin}).

The aforementioned three properties (functional equations, specialization and restriction formulas) show that the two sides of \eqref{e:imc} transform in strikingly parallel fashion. What is missing is a theorem that genuinely \emph{links} them. The first of our main results provides such a link by identifying, after twisting by $\chi$, the
relevant Selmer-theoretic characteristic ideal over $\Gamma_0$ in terms of twisted Hasse--Weil
$L$-values.

 \begin{mytheorem}\label{t:chi} Let the notations be as above. In $\Q_\chi\Lambda_{\Gamma_0}$, the ideal
\begin{equation}\label{e:5th}
\mathrm{CH}_{\Q_\chi\Lambda_{\Gamma_0}}(X_{L'_0}^{(\chi)})=(\mathsf c_\chi), \hbox{ for some element }\; \mathsf c_\chi\in\Q_\chi\Lambda_{\Gamma_0},
\end{equation}
such that for all $\omega\in\hat\Gamma_0$,
$$\omega(\mathsf c_\chi)=L_{A/K}(\omega^{-1}\chi^{-1},1).$$
\end{mytheorem}
Theorem~\ref{t:chi} will be proved in \S\ref{su:pftchi}. First we make a few remarks regarding the theorem.
\begin{remark}\label{r:rk}

\begin{enumerate}
\item When $K'=K$ (so that $\chi$ is trivial), the result is already known by \cite{lltt14a};
moreover, in \cite{tan24} this special case, together with the restriction identities
\eqref{e:rest} and \eqref{e:anrest}, is used to deduce \eqref{e:mu}. 
\item A noteworthy feature of \eqref{e:5th} is that the ideal
$\CH_{\Q_\chi\Lambda_{\Gamma_0}}(X_{L_0'}^{(\chi)})$ is independent of the auxiliary choice of
$K'$. 
\end{enumerate}
\end{remark}
Our second main theorem is an equivalent reformulation of Theorem~\ref{t:chi} that avoids the
auxiliary choice of $K'$.
It may be viewed as a strengthened `$\chi$-formula', expressed purely in terms of
specialization along the unramified direction. To state it, we first extend the specialization map $p^L_{L_0}$ to an $\O_\chi$-algebra
homomorphism
\[
  p^L_{L_0}:\O_\chi\Lambda_\Gamma \longrightarrow \O_\chi\Lambda_{\Gamma_0}.
\]
For $f\in\O_\chi\Lambda_\Gamma$, set
\[
  {}^\chi f \;:=\; p^L_{L_0}(f_\chi),
\]
so that $f\mapsto {}^\chi f$ defines an $\O_\chi$-algebra homomorphism
$\O_\chi\Lambda_\Gamma\to \O_\chi\Lambda_{\Gamma_0}$.

\begin{mytheorem}\label{t:1st}
With the above notation, one has the equality of ideals in $\O_\chi\Lambda_{\Gamma_0}$
\begin{equation}\label{e:1st}
  {}^\chi \CH_{\Lambda_\Gamma}(X_L) \;=\; ({}^\chi \mathscr L_{A/L}).
\end{equation}
\end{mytheorem}

Theorem~\ref{t:1st} will be proved in \S\ref{su:pf1st}.

 \subsection{The technical hypothesis on $\mu$-invariants}\label{su:hypo} 
Theorems~\ref{t:chi} and~\ref{t:1st} provide a link between the two sides of \eqref{e:imc} after twisting by a finite-order character and specializing
along the `unramified direction'. To pass from these equalities to the full equality of ideals in $\Lambda_\Gamma$, one requires a mild hypothesis controlling $\mu$-invariants under specialization, which we now formulate.

For a given $\Z_p^d$-extension $L/K$, set $\tilde L:=L\,L_0$, and write $\Gamma_0:=\Gal(L_0/K)$ and $\tilde\Gamma:=\Gal(\tilde L/K)$.
\begin{hypothesis}[Unramified $\mu$-minimality hypothesis]\label{h:h}
The $\mu$-invariant of $\mathscr L_{A/\tilde L}\in \Lambda_{\tilde\Gamma}$ coincides with that of 
$p^{\tilde L}_{L_0}\!\left(\mathscr L_{A/\tilde L}\right)\in \Lambda_{\Gamma_0}$.
\end{hypothesis}

To motivate the hypothesis, we first note that, for any intermediate $\Z_p$-extension $L'/K$ of $\tilde L/K$, one always has
\[
\mu\!\left(p^{\tilde L}_{L'}(\mathscr L_{A/\tilde L})\right)\ \geq\ \mu\!\left(\mathscr L_{A/\tilde L}\right).
\]
It is also known that equality holds for a large class of choices of $L'/K$ (see \cite[Lemma~6.2.5]{tan24}).
Accordingly, the above hypothesis asserts that $\mu\!\left(p^{\tilde L}_{L'}(\mathscr L_{A/\tilde L})\right)$ attains its minimal value at $L'=L_0$.
It is worth noting that since the specialization formula holds (Corollary~\ref{c:5.2.2}) and 
$\varrho_{\tilde L/L_0}\cdot \vartheta_{\tilde L/L_0}$ is not divisible by $p$ (see \S\ref{su:not}), in 
$\Lambda_{\Gamma_0}$  one has
\[
\mu\!\left(p^{\tilde L}_{L_0}(\mathscr L_{A/\tilde L})\right)
=\mu\!\left(\mathscr L_{A/L_0}\right).
\]
Hence the hypothesis holds if $\mu\!\left(\mathscr L_{A/L_0}\right)=0$; equivalently, if $\mu\!\left(X_{L_0}\right)=0$ by \eqref{e:mu}. 

The point of the hypothesis is that it allows one to lift the `$\chi$-formulas' to the full main conjecture 
\eqref{e:imc} in the setting of \S\ref{su:maint}, see Lemma \ref{l:root}.

\begin{mytheorem}\label{t:IMC-hyp}
Under the `unramified $\mu$-minimality hypothesis', the Iwasawa Main Conjecture \eqref{e:imc} holds for $A$ over
$L/K$.
\end{mytheorem}

The proof of Theorem~\ref{t:IMC-hyp} is given in \S\ref{s:imc}, where we combine the
$\chi$-formula of Theorem~\ref{t:1st} with the functional equation, specialization formulae, and
restriction compatibilities recalled above. 

For proving the theorem, a few reductions can be made. At first, in view of Remark \ref{r:rk} (1) and \eqref{e:nontor}, we shall assume that $L\not=L_0$ and $X_L$ torsion.

Also, since if \eqref{e:imc} holds for $A/\tilde L/K$, then it also holds for $A/L/K$ (see the first paragraph of \S\ref{su:induction}), we shall assume that $L=\tilde L$. The assumption implies $X_L$ is actually
torsion \cite[Theorem 2]{tan13} and \cite{ot09}.  Under this assumption, the $d=1$ case is excluded (and is already proven), while the $d=2$ case, as previously stated, follows from Theorem \ref{t:chi} and Theorem \ref{t:1st}.
It also turns out that the heart of the matter already lies in the case $d=2$. Once this case is in hand, the higher-rank case $d>2$ follows by an induction on $d$, in which one compares the two sides after specializing to suitably many intermediate $\Z_p^{d-1}$-extensions and then lifts the resulting equalities back to $\Lambda_\Gamma$ by a `Weierstrass preparation'
argument (see \S\ref{su:d>2}).

Since $X_L$ and $\mathscr L_{A/L}$ have the 
same $\mu$-invariant, for proving \eqref{e:imc}, it is convenient to work over $\Q_p\Lambda_\Gamma=\Lambda_\Gamma[1/p]$ instead, in other words, \eqref{e:imc} is equivalent to
\begin{equation}\label{e:imc'}
\mathrm{CH}_{\Q_p\Lambda_\Gamma}(\Q_pX_L)=\mathscr L_{A/L}\cdot \Q_p\Lambda_\Gamma.
\end{equation}

Finally, to emphasize that the unramified $\mu$-minimality hypothesis is far from vacuous, we establish in the Appendix
a result when $p>3$.
Using \cite[Theorem~6.3.1]{lst21} (together with the fact that $\deg(\Delta)$ is divisible by
$12$, cf.\ \cite[(9)]{tan95}), we show that for $n\gg 0$, the locus of elliptic curves over $K\otimes_{\F_q}\overline{\F}_q$ with global discriminant of degree $12n$ that are semistable everywhere and satisfy `$\mu=0$' is Zariski open and dense.
 
\subsection{Notations}\label{su:not} 
Let $S$ denote the ramification locus of $L/K$. Let $L'/K$ be an intermediate $\Z_p^e$-extension of $L/K$,
write $\Psi$ for $\Gal(L/L')$, and let $S'$ denote the ramification locus.

Let $\F_q$ be the constant field of $K$. In general, for a place $w$ of an algebraic extension $K'$ of $K$, let $\F_w$ denote the residue field and write $q_w:=|\F_w|$. If $K'/K$ is unramified at a place $v$, let $[v]_{K'/K}\in\Gal(K'/K)$ denote the Frobenius element.

\subsubsection{The global factor $\varrho_{L/L'}$}\label{ss:gbfactor} Since 
$A_{p^\infty}(L)$ is finite, by \cite[Definition 2.1.1]{tan24}, 
\begin{equation}\label{e:varrho}
\varrho_{L/L'}=
\begin{cases}
(1), & \text{if}\; e\geq 1;\\
(|A_{p^\infty}(K)|)^2, & \text{if}\; e=0. \\
\end{cases}
\end{equation}
This implies that if $L'\subsetneq L''\subsetneq L$, where $L''/K$ is an $\Z_p^f$-extension, then \begin{equation}\label{e:e<f}
\varrho_{L/L'}=p^{L''}_{L'}(\varrho_{L/L''})\cdot \varrho_{L''/L'}.
\end{equation}

\subsubsection{Local factors $\vartheta_{L/L',v}$}\label{ss:locfactor}
The factor
$\vartheta_{L/L'}$ is defined to be the product $\prod_v\vartheta_{L/L',v}$ for $v$ running through all places of $K$, and each $\vartheta_{L/L',v}$ is an ideal of $\Lambda_{\Gamma'}$.

Let ${\varPi_v}$ denote the group of the connected components of the
closed fiber (over $\bar\F_v$)
of the N$\acute{\text{e}}$ron model of $A/K_v$ and put
$$m_v:=|\varPi_v^{\Gal(\bar\F_v/\F_v)}|.$$

Suppose $A$ has good ordinary reduction $\bar A$ at $v$. Then $\bar A_{p^\infty}(\bar\F_v)$ is a $p$-divisible group of co-rank $1$ and the
action of the Frobenius substitution $\Frob_v\in \Gal(\bar\F_v/\F_v)$ on the Tate module $\mathrm{T}(\bar A_{p^\infty}(\bar\F_v))\simeq \Z_p$ has an eigenvalue $\alpha_v\in\Z_p^\times$.
Also, \cite[Corollary 4.37]{maz} states that the eigenvalues of the Frobenius endomorphism
$\tF_v:{\bar A}\longrightarrow {\bar A}$ over $\F_v$ are $\alpha_v$ and $\beta_v=q_v/\alpha_v$.
Hence $\alpha_v$ is a Weil $q_v$-number of weight 1.
In particular,
\begin{equation}\label{e:weilnumber}
\alpha_v^m\not=1,\;\;\text{for every}\;\; m\in\Z.
\end{equation}

Suppose $A$ has split multiplicative reduction at $v$. Let $Q_v\in K_v^*$ be the local Tate period such that
$A({\bar K}_v)\simeq \bar K_v^*/Q_v^\Z$. Consider the composition
$$\xymatrix{\mathcal R_v:Q_v^\Z \ar@{^(->}[r] & K_v^* \ar[r]^-{\mathsf R_v}  & \Gamma_v,}$$
where $\mathsf R_v$ is the local reciprocity map, and let $\overline{\mathcal R_v(Q_v^\Z)}\subset \Gamma_v$
denote the closure of $\mathcal R_v(Q_v^\Z)$ in the $p$-adic topology. Put
$$\mathfrak{w}_v=\mathrm{CH}_{\Z_p}(\Gamma_v/\overline{\mathcal R_v(Q_v^\Z)}),$$
which is zero when $d\geq 2$.

The local factor $\vartheta_{L/L',v}$ is defined as follows:
\begin{enumerate}
\item[(a)] Suppose $v\not\in S$. If  $\Psi_v\not=0$, $\vartheta_{L/L',v}=(m_v);$
otherwise, $\vartheta_{L/L',v}=(1)$.

\item[(b)] Suppose $v\in S$ and $A$ has good ordinary reduction at $v$. Then
$$\vartheta_{L/L',v}=\begin{cases}
(1-\alpha_v^{-1}\cdot [v]_{L'/K})\times
 (1-\alpha_v^{-1}\cdot [v]_{L'/K}^{-1}), & \text{ if } v\not\in S';\\
 (1), & \text{ otherwise}.
\end{cases}
 $$ 

\item[(c)] Suppose $v\in S$ and $A$ has split multiplicative reduction at $v$.
Then
$$
\vartheta_{L/L',v}=\begin{cases}
\Lambda_{\Gamma'}\cdot \mathfrak{w}_v, & \text{if}\ \ \Psi_v\not=0,\;\text{and}\;\; \Gamma'_v=0;\\
(\sigma_v-1), & \text{if}\ \ \Psi_v\not=0, \;\text{and}\;\; \Gamma'_v \;\text{is topologically generated by}\; \sigma_v
;\\
(1),& \text{otherwise.}
\end{cases}
$$
\item[(d)] Suppose $v$ is a non-split multiplicative place in $S$. Then
$$
\vartheta_{L /L^{\prime}, v}:= 
\begin{cases}
(2 m_{v}), & \text { if } \Psi_{v} \simeq \mathbb{Z}_{p}^{f}, f \geq 2,\Gamma_{v}^{\prime}=0,  
\\ 
(2 m_v), & \text{ if }  \Psi_{v} \simeq \mathbb{Z}_{p}, \Gamma_{v}^{\prime}=0,  \mathbb{F}_{q_{v}^{2}} \not \subset L_{v} ; \\
(m_{v}), & \text { if } \Psi_{v} \simeq \mathbb{Z}_{p},\Gamma_{v}^{\prime}=0,  \mathbb{F}_{q_{v}^{2}} \subset L_{v} ; 
\\ (1+[v]_{L^{\prime}/K}), & \text { if } \Psi_{v} \neq 0, v\not\in S', \Gamma_{v}^{\prime}\simeq \Z_p;  \\ 
(1+\sigma_v), & \text { if } \Psi_{v} \neq 0, v \in S', \Gamma'_v \;\text{is topologically generated by}\; \sigma_v ,  \mathbb{F}_{q_{v}^{2}} \subset L_{v}^{\prime} ; \\ 
(1), & \text { otherwise. }
\end{cases}
$$
\end{enumerate}

It follows directly from the above that if $L'\subsetneq L''\subsetneq L$, and $L''/K$ a $\Z_p^f$-extension, then 
\begin{equation}\label{e:thetaef}
\vartheta_{L/L'}=p^{L''}_{L'}(\vartheta_{L/L''})\cdot \vartheta_{L''/L'}.
\end{equation} 

Note that if $L_0\subset L$ and $L'=L_0$, then $S'=S$, $\Psi_v=0$, for $v\not\in S$, and $\Gamma'_v\not=0$, for $v\in S$. This implies that $\vartheta_{L/L'}$ is not divisible by $p$.


\subsubsection{The element $\dag_{A/L}$}\label{ss:dag} 
Let $S_{1} \subset S$ denote the subset consisting of places $v$ such that $\Gamma_{v} \simeq \mathbb{Z}_{p}$. For such $v$, we fix a topological generator $\sigma_{v}$ of $\Gamma_{v}$.
The element $\dag_{A/L}=\prod_v \dag_{A/L,v}$, where for each place $v$ of $K$, 
$\dag_{A/L, v} \in \Lambda_\Gamma$ is defined as follows:

\begin{enumerate}

\item $\dag_{A / L, v}=m_v$, if $v \notin S$ and $\Gamma_{v}$ is trivial.

\item $\dag_{A / L, v}=\lambda_{v}-\sigma_{v}$, if $v \in S_{1}$ and $A$ has split multiplicative reduction at $v$, or $A$ has non-split multiplicative reduction at $v\in S_{1}$  and $\F_{q_v^2}\subset L_v$.

\item $\dag_{A / L, v}=1$, otherwise.
\end{enumerate}

\begin{remark}\label{r:dag} The definition of $\dag_{A/L,v}$, $v\in S_1$, depends on the choice of $\sigma_v$,
while a different choice changes $\dag_{A/L,v}$ by multiplication by a unit in $\Lambda_\Gamma$, so that the ideal $(\dag_{A/L,v})$
is well-defined. It is known that $\dag_{A/L}\cdot \mathscr L_{A/L}$ is independent of such choices, see \cite[Definition 3.4.1]{tan24}. When $L_0\subset L$, it follows that $S_1=\emptyset$, and such ambiguity disappears.

\end{remark}

\subsubsection{The fudge factor $\star_\omega$}\label{ss:fudge} 

Let $\omega\in\hat\Gamma$ and denote by $D_\omega$ its global Artin conductor. Let $S_o\subset S$ (resp. $S_m\subset S$) denote the subset consisting of places where $A$ has good ordinary (resp. multiplicative) reduction. For $v\in S_o$ let $\alpha_v$ be the eigenvalue in
\S\ref{ss:locfactor}. For $v\in S_m$,
let $\lambda_v=1$, if $A$ has split multiplicative reduction at $v$; otherwise, put $\lambda_v=-1$. 

We have
\begin{equation}\label{e:fudge}
\star_\omega=\omega(\dag_{A / L})^{-1} \cdot \alpha_{D_{\omega}}^{-1} \cdot \tau_{\omega} \cdot q^{\frac{\deg(\Delta)}{12}+\kappa-1} \cdot \Xi_{S, \omega},
\end{equation}
where
\begin{enumerate}
\item 
$$\alpha_{D_\omega}:=\prod_{v\in S_o}\alpha_v^{\ord_v D_\omega}\cdot \prod_{v\in S_m\cap \mathrm{Supp}(D_\omega)}\lambda_v^{\ord_v D_\omega-1}.$$
\item $\tau_\omega$ is the Gauss sum defined in \cite[\S 3.1.2]{tan24}.
\item $\Delta$ is the global discriminant of $A/K$ and $\kappa$ is the genus of $K$.
\item According to \cite[\S 3.2.2]{tan24},
$$\Xi_{S,\omega}=\prod_{\substack{v \in S_{m} \\ v \notin \mathrm{Supp}(D_{\omega})}}(\lambda_{v}-\omega([v]_{L^{Ker(\omega)}/K})^{-1}) \prod_{\substack{v \in S_{o} \\ v \notin \mathrm{Supp}(D_{\omega})}}(1-\alpha_{v}^{-1} \omega([v]_{L^{Ker(\omega)}/K}))(1-\alpha_{v}^{-1} \omega([v]_{L^{Ker(\omega)}/K})^{-1}).
$$
\end{enumerate}

\subsection*{Organization of the article}
The remainder of the article is organized as follows.

In \S\ref{s:chif} we prove the $\chi$-formulae (Theorems~\ref{t:chi} and~\ref{t:1st}), which provide the
key link between the analytic and algebraic sides of \eqref{e:imc} after twisting and
specializing along the unramified direction.

In \S\ref{s:imc} we deduce the Iwasawa Main Conjecture from these formulae, under the
`unramified $\mu$-minimality hypothesis' of \S\ref{su:hypo}; in particular, we explain the reduction from $d>2$ to the rank-two case.

In \S\ref{s:isotrivial} we treat the remaining isotrivial case, which requires a separate argument.
Here we analyze the relevant torsion subgroups and the associated field of twist, and then
complete the proof of \eqref{e:imc} in this setting.

The Appendix \ref{s:app} establishes a result showing that the `$\mu$-minimality
hypothesis' holds on a Zariski open dense locus in the relevant moduli space.

\section{The $\chi$-formula}\label{s:chif} 

In this section, we prove Theorem \ref{t:chi} and Theorem \ref{t:1st}.  Recall that $L_0':=K'\cdot K^{(p)}_\infty$.
Denote $\Psi_0:=\Gal(L/L_0)$, $\Psi'_0:=\Gal(L/L_0')$.
The restriction of Galois action identifies $G_\infty:=\Gal(L_0'/L_0)$ with $G$.
\usetikzlibrary{arrows.meta}

\begin{figure}[ht]
  \centering
  \begin{tikzpicture}[>=latex, node distance=3cm, font=\small]
    \node (L)   at (0,  0)   {$L$};
    \node (L1)  at (3,  -1)  {$L_1$};
    \node (L0)  at (0, -3)   {$L_0$};
    \node (K)   at (3, -4)   {$K$};

    \node (L')   at (0, -1.5)   {$L_0'$};
    \node (Kchi) at (3, -2.5)   {$K'$};

    \draw[<->, bend left=20] (L)   to node[above] {$\Psi_1$}   (L1);
    \draw[<->, bend left=20] (L')  to node[above] {$\Psi'_1$}  (Kchi);
    \draw[<->, bend left=20] (K)   to node[below] {$\Gamma_0$} (L0);

    \draw[<->, bend left=20] (K)   to node[left]  {$G$}        (Kchi);
    \draw[<->, bend left=20] (L0)  to node[left]  {$\Psi_0$}   (L);

    \draw[<->, bend right=20] (L') to node[right] {$\Psi'_0$}  (L);
    \draw[<->, bend left=20] (L1)  to node[right] {$\Gamma_1$} (K);
    \draw[<->, bend left=20] (L')  to node[right] {$G_\infty$} (L0);

    \draw (L)  -- (L');
    \draw (L') -- (L0);
    \draw (L1) -- (Kchi) -- (K);
    \draw (L') -- (Kchi);
    \draw (L)  -- (L1);
    \draw (L0) -- (K);
  \end{tikzpicture}
  \caption{Field extension diagram}
  \label{fig:field-diagram-setup}
\end{figure}
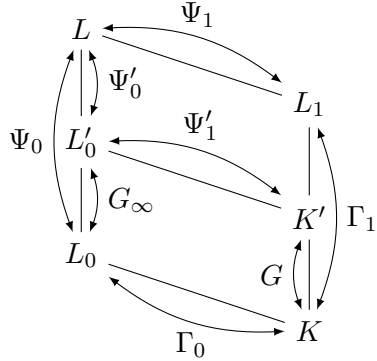

\subsection{The proof of Theorem \ref{t:chi}}\label{su:pftchi}

We will need the following lemmas.

\begin{lemma}\label{LemFin}
Let $\chi:G\to \Q_\chi^\times$ be a character with values in a finite extension
$\Q_\chi/\Q_p$.
Let $M$ be a module over $\Lambda_{\Gamma_0\times G}$.
Assume that $M$ is finitely generated (resp.\ finitely generated and torsion, resp.\ finitely generated and pseudo-null)
as a $\Lambda_{\Gamma_0}$-module. Then $M^{(\chi)}$ is finitely generated (resp.\ finitely generated and torsion, resp.\ trivial)
as a module over $\Q_\chi\Lambda_{\Gamma_0}$.
\end{lemma}

\begin{proof}
Let \( R := \Lambda_{\Gamma_0} \).
There is a natural surjection
\[
M \otimes_{\mathbb{Z}_p} \Q_\chi \longrightarrow M^{(\chi)}
\]
given by \( m \otimes \lambda \mapsto m \otimes \lambda \), whose kernel is the \( \Q_\chi \)-submodule generated by the relations
\[
g \cdot m \otimes \lambda - m \otimes \chi(g)\lambda, \quad \text{for all } g \in G, m \in M, \lambda \in \Q_\chi.
\]

Now, since \( M \) is finitely generated over \( R \), the scalar extension \( M \otimes_{\mathbb{Z}_p} \Q_\chi \) is finitely generated over \( R\otimes_{\mathbb{Z}_p}\Q_\chi\), and the quotient \( M^{(\chi)} \) is therefore also finitely generated over this ring.

If \( M \) is torsion over \( R \), then for each \( m \in M \), there exists a nonzero \( f \in R \) such that \( f \cdot m = 0 \). It follows that \( f \cdot (m \otimes 1) = 0 \) in \( M^{(\chi)} \), so \( M^{(\chi)} \) is torsion over \( \Q_\chi \otimes R \).

Finally, if $M$ is pseudo-null, then it is a finite group since $\Lambda_{\Gamma_0}\simeq \Z_p[[T]]$ has Krull dimension 2 as a ring so that  \( M \otimes_{\mathbb{Z}_p} \Q_\chi=0 \). Hence, its quotient \( M^{(\chi)} \) must be also trivial. \qedhere
\end{proof}

\begin{lemma}\label{crys-descent}
Let \( C/\mathbb{F}_q \) be a smooth proper curve, \( Z \subset C \) a divisor with normal crossings, and let \( f: C' \to C \) be a finite Galois cover with group \( G \), ramified only over \( Z \). Let \( C'^\sharp \) and \( C^\sharp \) denote the log-smooth schemes associated with \( Z' := f^{-1}(Z) \) and \( Z \), respectively.

Let \( E \) be a finite locally free $F$-log-crystal on \( C^\sharp/W(\mathbb{F}_q) \). Consider the tower of constant field extensions:
\[
C_n^{'\sharp} := C^{'\sharp} \times_{\mathbb{F}_q} \mathbb{F}_{q^{p^n}}, \qquad C_\infty^{'\sharp} := C^{'\sharp} \times_{\mathbb{F}_q} \mathbb{F}_{q^{p^\infty}},
\]
with Galois group \( \Gamma_0 \). Let \( N \in \mathbb{Z}_{\geq 0} \) and write \( C_N^{'\sharp} := C^{'\sharp} \times_{\mathbb{F}_q} \mathbb{F}_{q^{p^N}} \).

Define:
\begin{align*}
P^i_\infty &:= \coh^i \left( \varinjlim_n R\Gamma_{\log\rm{-crys}}(C'_n{}^\sharp/W(\mathbb{F}_{q^{p^n}}), E) \otimes^{\mathbb{L}} \mathbb{Q}_p/\mathbb{Z}_p  \right)^\vee, \\
P^i_N &:= \coh^i\left(  R\Gamma_{\log\rm{-crys}}(C'_N{}^\sharp/W(\mathbb{F}_{q^{p^N}}), E) \otimes^{\mathbb{L}} \mathbb{Q}_p/\mathbb{Z}_p  \right)^\vee.
\end{align*}

Then \( P^i_\infty \) is naturally a \( \Lambda_{\Gamma_0 \times G} \)-module, and \( P^i_N \) is a \( \mathbb{Z}_p[\Gamma_{0,N} \times G] \)-module, where $\Gamma_{0,N}$ is the finite quotient of $\Gamma_0$ isomorphic to $\mathbb{Z}/p^N$. Let \( \psi = \omega \cdot \chi \in \widehat{\Gamma_0 \times G} \) be a finite-order character factoring through \( \Gamma_{0,N} \times G \), so that \( \mathbb{Q}_\psi := \mathbb{Q}_\omega \cdot \mathbb{Q}_\chi \).

Then, we have a canonical Frobenius-equivariant isomorphism:
\[
(P^i_\infty)^{(\psi)} := P^i_\infty \otimes_{\mathbb{Z}_p[[\Gamma_0 \times G]]} \mathbb{Q}_\psi
\cong
P^i_N \otimes_{\mathbb{Z}_p[\Gamma_{0,N} \times G]} \mathbb{Q}_\psi =: (P^i_N)^{(\psi)}.
\]
\end{lemma}

\begin{proof}

First, note that by the proof of \cite[Lemma 3.1.2]{lltt16}, we have a Frobenius and $\Gamma_{0,N}$ equivariant isomorphism 
\[
P^i_\infty \otimes_{\Lambda_{\Gamma_0}} \mathbb{Q}_p[\Gamma_{0,N}] \xrightarrow{\sim} P^i_N[1/p].
\]
This isomorphism is moreover \( G \)-equivariant. Indeed, the cover \( f: C' \to C \) is a finite Galois cover with group \( G \), so each \( C_n^{'\sharp} \) admits a canonical right action by \( G \) via log-scheme automorphisms over \( C \).
The log-crystalline cohomology is functorial under log-scheme automorphisms. Thus, the action of \( G \) on each \( C_n^{'\sharp} \) induces an action on the cohomology groups. The tower \( \{C_n^{'\sharp}\} \) is formed by constant field extensions \( \mathbb{F}_{q^{p^n}}/\mathbb{F}_q \), which are linearly disjoint from the extension \( C'/C \), so the actions of \( \Gamma_0 \) and \( G \) commute. Hence, the transition maps in the inverse system defining \( P^i_\infty \) are \( G \)-equivariant, and so the inverse limit inherits a natural \( G \)-action. Finally, the scalar extension from \( \Lambda_{\Gamma_0} \) to \( \mathbb{Q}_p[\Gamma_{0,N}] \) affects only the \( \Gamma_0 \)-action, and leaves the \( G \)-structure untouched and we can conclude the required $G$-equivariance.

By tensoring both sides of the $\Gamma_{0,N} \times G$-equivariant isomorphism $P^i_\infty \otimes_{\Lambda_{\Gamma_0}} \mathbb{Q}_p[\Gamma_{0,N}] \xrightarrow{\sim} P^i_N[1/p]$ with $\otimes_{\mathbb{Q}_p[\Gamma_{0,N} \times G]} \mathbb{Q}_\psi$ and applying the next lemma to the left-hand side, we obtain the claimed isomorphism:
\[
(P^i_\infty)^{(\psi)} \xrightarrow{\sim} (P^i_N)^{(\psi)}.
\]

\end{proof}

  \begin{lemma}
Let $M$ be a $\Lambda_{\Gamma_0 \times G}$-module. Set \[
N \;:=\; M \otimes_{\Lambda_{\Gamma_0}} \mathbb{Q}_p[\Gamma_{0,N}], \]
where $M$ is viewed as a $\Lambda_{\Gamma_0}$-module by restriction of scalars. Then $N$ has a natural $\mathbb{Q}_p[\Gamma_{0,N} \times G]$-module structure given by
\[
(\bar\gamma,h)\cdot(m\otimes \bar\gamma') \;=\; (h\cdot m)\otimes \bar\gamma\,\bar\gamma'\qquad(\bar\gamma,\bar\gamma'\in\Gamma_{0,N},\ h\in G,\ m\in M).
\]
Let $\psi:\Gamma_{0,N} \times G \to \mathbb{Q}_\psi^*$ be a character. Then there is a canonical isomorphism
\[
M \otimes_{\Lambda_{\Gamma_0 \times G}} \mathbb{Q}_\psi
\;\cong\;
N \otimes_{\mathbb{Q}_p[\Gamma_{0,N}\times G]} \mathbb{Q}_\psi,
\]
where in both tensor products the right tensor factor is taken via the character $\psi$.
\end{lemma}
                                
\begin{proof}
There is a canonical isomorphism of algebras
\[
\Lambda_{\Gamma_0\times G}\otimes_{\Lambda_{\Gamma_0}}\mathbb{Q}_p[\Gamma_{0,N}]
\;\cong\; \mathbb{Q}_p[\Gamma_{0,N}\times G],
\]
so, for the given $\Lambda_{\Gamma_0\times G}$-module $M$,
\[
M\otimes_{\Lambda_{\Gamma_0}}\mathbb{Q}_p[\Gamma_{0,N}]
\;\cong\;
M\otimes_{\Lambda_{\Gamma_0\times G}}\mathbb{Q}_p[\Gamma_{0,N}\times G].
\]
Hence, by associativity of tensor products,
\[
\begin{aligned}
\bigl(M\otimes_{\Lambda_{\Gamma_0}}\mathbb{Q}_p[\Gamma_{0,N}]\bigr)
\otimes_{\mathbb{Q}_p[\Gamma_{0,N}\times G]}\mathbb{Q}_\psi
&\;\cong\;
\bigl(M\otimes_{\Lambda_{\Gamma_0\times G}}\mathbb{Q}_p[\Gamma_{0,N}\times G]\bigr)
\otimes_{\mathbb{Q}_p[\Gamma_{0,N}\times G]}\mathbb{Q}_\psi\\[4pt]
&\;\cong\;
M\otimes_{\Lambda_{\Gamma_0\times G}}\mathbb{Q}_\psi,
\end{aligned}
\]
which is exactly the claimed isomorphism.
\end{proof}

\begin{lemma}\label{Leminterpol} Let $f: U' \to U$ be a finite \'etale Galois covering of smooth curves over $\mathbb{F}_q$ with Galois group $G$ abelian. Let $E$ be an overconvergent $F$-isocrystal on $U$ with coefficients in $\Q_p$, and $E' := f^*E$ its pullback to $U'$. Let $\chi: G \to {\bar\Q_p}^*$ be a 1-dimensional character with field of values $\Q_\chi$. Then such $\chi$ corresponds to a $p$-adic representation of the \'etale fundamental group $\pi^{et}_1(U)$ with finite local monodromy at each place at infinity. We denote $U(\chi)^\dagger$ the associated unit-root overconvergent $F$-isocrystal on $U/\Q_\chi$ by \cite{Tsu98}.
There is a canonical Frobenius-compatible isomorphism:
\[
\coh^i_{\rig,c}(U'/\Q_p, E')^{(\chi)} \simeq \coh^i_{\rig,c}(U / \Q_\chi, E \otimes U(\chi)^\dagger).
\]
\end{lemma}

\begin{proof}
Because $f$ is finite \'etale with Galois group $G$, the pullback $E' = f^*E$ carries a $G$-action, and so the cohomology group $\coh^i_{\rig,c}(U'/\Q_p, E')$ is a finite-dimensional $\Q_p[G]$-module.

By definition, its isotypic quotient is:
\[
\coh^i_{\rig,c}(U'/\Q_p, E')^{(\chi)}=\coh^i_{\rig,c}(U'/\Q_p, E') \otimes_{\Q_p[G]} \Q_\chi.
\]

Now, by \cite[Proposition 5.15]{TV19}, we have a canonical Frobenius-compatible isomorphism:
\[
\coh^i_{\rig,c}(U'/\Q_p, E') \otimes_{\Q_p[G]} \Q_\chi \cong \coh^i_{\rig,c}(U / \Q_\chi, E \otimes U(\chi)^\dagger).
\]

Thus,
\[
\coh^i_{\rig,c}(U'/\Q_p, E')^{(\chi)} \cong \coh^i_{\rig,c}(U/\Q_\chi, E \otimes U(\chi)^\dagger),
\]
as required.
\end{proof}

In our context, the isocrystal $E$ arises as the overconvergent realization of the log Dieudonn\'{e} crystal $D^{\log}(A)(-Z)$, associated with $A/K$ and the divisor $Z$ supported on the bad reduction and ramification loci.

From now on, we denote
\[
P^i_\infty := \coh^i\left(\varinjlim_n R\Gamma_{\log\text{-crys}}(C_n^{'\sharp}/W(\mathbb{F}_{q^{p^n}}), D^{\log}(A)(-Z)) \otimes^{\mathbb{L}} \Q_p/\Z_p \right)^\vee,
\]
which comes equipped with Frobenius operators \( \varphi_i \) acting on the cohomology,

and
\[
N^i_{\infty} := \coh^i\left( \varinjlim_n R\Gamma(C'_n, \mathcal{S}_{D^{\log}(A)(-Z)}) \otimes^{\mathbb L} \mathbb{Q}_p/\mathbb{Z}_p \right)^\vee,
\]
where $\mathcal{S}_{D^{\log}(A)(-Z)}$ is the syntomic complex of $D^{\log}(A)(-Z)$ (see \cite[\S 2.1.3]{lltt16}).

Note that our notation is different from \cite{lltt16}: our $P^i_\infty$ (resp. $N^i_\infty$) is denoted $(M^i_{2,\infty})^\vee$ (resp. $(N^i_\infty)^\vee$) in \cite{lltt16}.

Both groups $N^i_\infty$ and $P^i_\infty$ are naturally (not necessarily finitely generated) $\Lambda_{\Gamma_0 \times G}$-modules. However, as $\Lambda_{\Gamma_0}$-modules,  $N^i_\infty$ is finitely generated torsion, whereas $P^i_\infty$ is finitely generated, since they correspond to the groups $N_\infty$ and $P_\infty$ calculated for $A':=A \times_K K'$ over the extension $L'/K'$ as proved in \cite{lltt16}. Indeed, note that 
$$P^i_\infty\simeq \coh^i\left( \varinjlim_n R\Gamma_{\log\text{-crys}}(C_n^{'\sharp}/W(\mathbb{F}_{q^{p^n}}), D^{\log}(A')(-Z')) \otimes^{\mathbb{L}} \Q_p/\Z_p \right)^\vee.$$

\begin{lemma}\label{l:long}
We have a long exact sequence of finitely generated $\Q_\chi\Lambda_{\Gamma_0}$-modules:
\[
\cdots \to (P^i_\infty)^{(\chi)} \xrightarrow{1 - \varphi_i} (P^i_\infty)^{(\chi)} \to (N^i_\infty)^{(\chi)} \to (P^{i+1}_\infty)^{(\chi)} \xrightarrow{1 - \varphi_{i+1}} \cdots.\]

Moreover, the groups $(N^i_\infty)^{(\chi)}$ are torsion modules and all terms in the long exact sequence vanish for $i \geq 3$.
\end{lemma} 

\begin{proof} The long exact sequence exists before applying the exact functor $\chi$-isotypic component and is constructed exactly as the Pontryagin dual of the long exact sequence \cite[(7)]{lltt16}, and after killing $p$, noting that the Lie cohomology (denoted $L^i_\infty$ in \cite{lltt16}) becomes trivial after inverting $p$. Note also that our extension is $L'/K$ and not $K^{(p)}_\infty/K$ as in \cite{lltt16} but the arguments to construct the long exact sequence are identical. The remaining assertions are direct consequences of Lemma \ref{LemFin} and \cite[Theorem 2.1.2]{lltt16}.
\end{proof}

We define the $\chi$-twisted $p$-adic $L$-function by: ($Q(R)$ denotes the field of fractions of $R$)
\[
L_\chi := \prod_{i=0}^2 \det\left(1 - \varphi_i \mid (P^i_\infty)^{(\chi)}\otimes_{\Q_\chi\Lambda_{\Gamma_0}} Q(\Q_\chi\Lambda_{\Gamma_0})\right)^{(-1)^{i+1}} \in Q(\Q_\chi {\Lambda_{\Gamma_0}})^*.
\]

Note that $L_\chi$ is well defined: The modules $ (N^i_\infty)^{(\chi)}$ are torsion over $\Q_\chi\Lambda_{\Gamma_0}$, hence the maps $ 1 - \varphi_i$ become isomorphisms after tensoring with the total ring of fractions $Q(\Q_\chi\Lambda_{\Gamma_0})$. Thus, each determinant \( \det(1 - \varphi_i \mid (P^i_\infty)^{(\chi)})\otimes_{\Q_\chi\Lambda_{\Gamma_0}} Q(\Q_\chi\Lambda_{\Gamma_0})) \) is not zero, and the alternating product $L_\chi$ defines an element of $Q(\Q_\chi\Lambda_{\Gamma_0})^*$.

\begin{theorem}\label{t:lchi}
Let \( \psi = \omega \cdot \chi \in \widehat{\Gamma_0 \times G} \) where $\omega$ is a  finite-order character of $\Gamma_0$. Then:
\[
\omega(L_\chi) = \prod_{i=0}^2 \det\left(1 - \varphi_i \mid \coh^i_{\mathrm{rig},c}(U, D^\dagger(A) \otimes U(\psi^{-1})^\dagger) \right)^{(-1)^{i+1}} = L_Z(A, \psi^{-1}, 1).
\]

\end{theorem}

\begin{proof} Our proof follows closely the proof of \cite[Lemma 3.1.2]{lltt16}. First, note that by construction, 
\[
\omega(L_\chi) = \prod_{i=0}^2 \det\left(1 - \varphi_i \mid (P^i_\infty)^{(\chi)}\otimes_{\Q_\chi\Lambda_{\Gamma_0}}\Q_\chi\Q_\omega \right)^{(-1)^{i+1}} \]
where the tensor product $\otimes_{\Q_\chi \Lambda_{\Gamma_0}} \Q_\chi \Q_\omega$ is done via the ring homomorphism induced by $\omega$. Let $K_\psi/K$ be a finite Galois extension through which $\psi$ factors. This extension can be written $K_\psi=K'K_N$ for some intermediate extension $K_N$ of $L_0/K$ and we denote $\Gamma_{0,N}:=\mathrm{Gal}(K_N/K)$. The field $K_\psi$ is the function field of the proper smooth curve $C'_N$ and we denote $U'_N$ the open subset where we have removed the places in $C'_N$ above $Z$. By Lemma \ref{crys-descent}, the object $(P^i_\infty)^{(\chi)} \otimes_{\Q_\chi \otimes \Lambda_{\Gamma_0}} \Q_\chi \Q_\omega = (P^i_\infty)^{(\psi)}$ admits a canonical isomorphism with $(P^i_{N})^{(\psi)}$. By the proof of \cite[Lemma 3.1.2]{lltt16}, we have an isomorphism
$$P^i_N[1/p]\simeq \coh^i_{rig,c}(U'_N,D^\dagger(A))^\vee.$$

Hence, we deduce from Lemma \ref{Leminterpol}: 
\[
(P^i_{N})^{(\psi)}\simeq (\coh^i_{\rig,c}(U, D^\dagger(A) \otimes U(\psi^{-1})^\dagger))^\vee
\]
Then, since determinants are preserved under duality (because duality corresponds to transposition, and $\det(f) = \det(f^T)$), taking the alternating product of Frobenius determinants gives:
\[
\omega(L_\chi) = \prod_{i=0}^2 \det\left(1 - \varphi_i \mid \coh^i_{\rig,c}(U, D^\dagger(A) \otimes U(\psi^{-1})^\dagger) \right)^{(-1)^{i+1}} = L_Z(A, \psi^{-1}, 1), 
\]
where the second equality holds by \cite[(33)]{lltt16}.
\end{proof}

Following the formulation of \cite[(17)]{lltt16}, and noting that the Lie algebra cohomology groups become trivial over \(\Q_\chi\), we define the $\chi$-twisted characteristic element \(f_\chi\) as:
\[
f_\chi := \frac{f\left((N_\infty^1)^{(\chi)}\right)}{f\left((N_\infty^0)^{(\chi)}\right) \cdot f\left((N_\infty^2)^{(\chi)}\right)} \in Q(\Q_\chi\Lambda_{\Gamma_0})^*/(\Q_\chi\Lambda_{\Gamma_0})^*,
\]
where we have denoted for a finitely generated torsion \(\Q_\chi \Lambda_{\Gamma_0}\)-module $M$, $f(M)$ its characteristic element.

\begin{lemma}\label{l:lchixchi}
We have the equality of principal ideals in $\Q_\chi\Lambda_{\Gamma_0}$:
$$
(L_\chi)=(f_\chi) = \mathrm{CH}_{\Q_\chi \otimes \Lambda_{\Gamma_0}}(X_{L_0'}^{(\chi)}).
$$

\end{lemma}

\begin{proof} The first equality results from the long exact sequence of Lemma \ref{l:long} and of \cite[Lemma 3.1.8]{lltt16}. In order to prove the second equality, first note that under our hypothesis on $A/K$, the group $A({L_0'})[p^\infty]$ is finite. Now we can prove exactly as in \cite[(19)]{lltt16} that we have an exact sequence
$$0\to X_{L_0'}\to N^1_\infty\to \mathcal M_\infty\to A({L_0'})[p^\infty]^\vee\to N^0_\infty\to 0.$$
We deduce from Lemma \ref{LemFin} that $(N^0_\infty)^{(\chi)}=0$ and so its characteristic element is 1. The term $\mathcal M_\infty$ (denoted $\mathcal M_\infty^\vee$ in \cite{lltt16}) is trivial in the case of an elliptic curve, as easily deduced from \cite[Lemma 2.2.1]{lltt16}, if we don't have good reduction places in $Z$. Otherwise, the term $\mathcal M_\infty$ also has direct summand terms of the form $$\mathcal M_{\infty,v}:= \varprojlim_n [\oplus_{w_n \in C'_{n}, w_n|v} A_v(k_{w_n})^\vee],$$ where $v$ is a place of good reduction, ramified in ${L_0'}/K$. Then the inertia at $v$ acts trivially on $\mathcal M_{\infty,v}$ because the action of $I_v$ on $A_v(k_{w_n})$ is trivial.
However, if we are taking the $\chi$ component where $v$ is ramified, it means that $\chi$ will act nontrivially on $I_v$.
In particular, for such $v$, $\mathcal M_{\infty,v}^{(\chi)}$ is zero. Hence, the characteristic element of $\mathcal M_\infty$ is a unit. Finally, by \cite[Lemma 2.2.3]{lltt16} the group $N^2_\infty$ is trivial as a $\Lambda_{\Gamma_0}$-module and we deduce from Lemma \ref{LemFin}, that $(N^2_\infty)^{(\chi)}$ must also be trivial. 
\end{proof}

Finally, the proof of Theorem \ref{t:chi} is reduced to the following lemma.

\begin{lemma}\label{l:sharp}
We have the equality of principal ideals in $\Q_\chi\Lambda_{\Gamma_0}$:

\[
(L_\chi)=(\mathsf c_\chi),\]
where $\mathsf c_\chi:=P_{\chi^{-1}}(q^{-1}\mathsf F_q^{-1})$, and where $\mathsf{F}_q$ denotes the image of the Frobenius substitution $x\mapsto x^q$ under the isomorphism
$\Gal(\F_{q^{p^\infty}}/\F_q)\longrightarrow \Gamma_0$
and $P_{\chi^{-1}}(T)\in \O_\chi[T]$, the polynomial such that $P_{\chi^{-1}}(q^{-s})=L_A(\chi^{-1},s)$. 
\end{lemma}

\begin{proof}
Note that $\mathsf c_\chi$ is an element of $\Q_\chi\Lambda_{\Gamma_0}$ satisfying the interpolation formulas 
$$\omega(\mathsf c_\chi)=L(A,\omega^{-1}\chi^{-1},1),$$
for any $\omega\in \hat \Gamma_0$,
while 
$$\omega(L_\chi)=L_Z(A,(\omega.\chi)^{-1},1).$$ 
but these two values differ by the term $\omega(\rho)$, where we denote $\rho$ the element of $\Q_\chi\Lambda_{\Gamma_0}$:
$$\rho:=\prod_{v\in Z_{unr}} \rho_v^\sharp=\prod_{v \in Z_{unr}} (1-\lambda_v\cdot \chi([v]_{K'/K})^{-1}\cdot [v]_{L_0/K}\cdot q_v^{-1})^\sharp,$$
where $Z_{unr}$ are the places in $Z$, which are unramified in the extension $K'/K$.
So we have to prove that $\rho_v$ is a unit in $\Q_\chi\Lambda_{\Gamma_0}$. 
Denoting $s:=[v]_{L_0/K}-1$, which is in the augmentation ideal of $\Lambda_{\Gamma_0}$, we can rewrite 
$$\rho_v=q_v^{-1}\cdot ((q_v-\lambda_v\cdot \chi([v]_{K'/K})^{-1})-\lambda_v\cdot \chi([v]_{K'/K})^{-1}\cdot s).$$ 
But $q_v-\lambda_v\cdot \chi([v]_{K'/K})^{-1}=q_v\pm \chi([v]_{K'/K})^{-1}$ is a unit in $\O_\chi$ so 
$$\rho_v\in \Q_\chi^*\O_\chi[[T]]^*=(\Q_\chi\Z_p[[T]])^*.$$ 
\end{proof}

\subsection{The characteristic ideal of $\Q_\chi X_L/(\psi_0-\chi(\psi_0))$}\label{su:psio}
Let $\psi_0\in\Psi_0$ be a topological generator and let $\bar\psi_0$ be its image in $G_\infty$.
Instead of the decomposition $\Gamma_0'=\Gamma_0\times  G$ used before,
from now on, we shall use $\Gamma_0'=\Psi_1'\times G_\infty$, so that 
$\Q_\chi X_{{L_0'}}/(\bar\psi_0-\chi(\bar\psi_0))$ is a $\Psi_1'$-module and
is viewed as a $\Gamma_0$-module via the identification $\xymatrix{\Psi_1' \ar[r]^-\sim & \Gamma_0}$. Here $\chi$ is viewed as a character on $\Gamma'$. Similarly, $\Q_\chi X_L/(\psi_0-\chi(\psi_0))$
is a $\Gamma_0$-module.

Since $X_{L_0'}$ is torsion over $\Lambda_{\Psi'_1}$ \cite{ot09},
we have
\[
\dim_{\Q_p}\bigl(\Q_p\otimes_{\Z_p} X_{L_0'}\bigr)<\infty.
\]
Equivalently, $\Q_\chi X_{L_0'}=\Q_\chi\otimes_{\Z_p}X_{L_0'}$ is a finite-dimensional
$\Q_\chi$-vector space. Hence, as a $\Q_\chi[G_\infty]$-module, it decomposes as a finite
direct sum of isotypic components (eigenspaces) indexed by the characters of $G_\infty$.
Thus, the projection $\Q_\chi X_{L_0'}\longrightarrow X_{L_0'}^{(\chi)}$ induces the equality
\begin{equation}\label{e:(chi)}
\mathrm{CH}_{\Q_\chi\Lambda_{\Gamma_0}}(\Q_\chi X_{{L_0'}}/(\bar\psi_0-\chi(\bar\psi_0)))= \mathrm{CH}_{\Q_\chi\Lambda_{\Gamma_0}}( X_{L_0'}^{(\chi)})=(\mathsf c_\chi).
\end{equation}

In Lemma \ref{l:xlpsio} below we take the crucial step to deducing Theorem \ref{t:1st} from Theorem \ref{t:chi}, by relating the characteristic ideals of $\Q_\chi X_{{L_0'}}/(\bar\psi_0-\chi(\bar\psi_0))$ and $\Q_\chi X_L/(\psi_0-\chi(\psi_0))$.

For simplicity, we replace $\hat G$ by its subgroup generated by $\chi$. Thus, $v$ is unramified over ${L_0'}/K$ if and only if $\chi$ is unramified at $v$. 

For $v\in S$ but unramified over ${L_0'}/K$,
we define $\diamondsuit_{v,\chi}$ as follows. 
If $v$ is a good ordinary place, put
$$\diamondsuit_{v,\chi}=(1-\alpha_v^{-1} \cdot \chi([v]_{K'/K})^{-1}\cdot [v]_{L_0/K}^{-1})(1-\alpha_v^{-1} 
\cdot\chi([v]_{K'/K})\cdot [v]_{L_0/K}).$$
Here $\alpha_v$ is the $p-$adic unit Weil $q_v$-number associated to the reduction of $A$ at $v$.

If $v$ is a split multiplicative place, put
$$\diamondsuit_{v,\chi}=1-\chi([v]_{K'/K})^{-1}\cdot [v]_{L_0/K}^{-1}.$$

If $v$ is a non-split multiplicative place, put
$$\diamondsuit_{v,\chi}=
\begin{cases}
1, & \text{ if } p\not=2;\\
-1-\chi([v]_{K'/K})^{-1}\cdot [v]_{L_0/K}^{-1}, & \text{ if } p=2.
\end{cases}$$

If $v\in S$ and is ramified over $L'_0/K$, define
$$\diamondsuit_{v,\chi}=1.$$

\begin{lemma}\label{l:xlpsio}As ideals in $\Q_\chi\Lambda_{\Gamma_0}$,
$$\mathrm{CH}_{\Q_\chi\Lambda_{\Gamma_0}}(\Q_\chi X_L/(\psi_0-\chi(\psi_0))) =\prod_{v\in S}(\diamondsuit_{v,\chi})\cdot (\mathsf c_\chi).$$

\end{lemma}
{The proof is completed in \S\ref{ss:local}.}
\subsubsection{Specialization}\label{ss:sp}
Throughout \S2.2.1 and \S2.2.2, we shall rely on the references \cite{tan10}, \cite{tan13}, and \cite{tan24}, where only the structure of $\Psi_0'$, which in our case is isomorphic to $\Z_p$, is relevant, rather than the structure of $\mathrm{Gal}({L'_0}/K)$. In particular, all these references apply directly to our situation.

In the following discussion, let $v$, $v'$, $w$, $w'$ and $\tilde w$ 
be respectively places of $K$, $K'$, $L_0$, $L'_0$ and $L$ such that $\tilde w\mid w'\mid w\mid v$ and
$w'\mid v'\mid v$.
Since no place of $K$ splits completely over $L_0$, 
$$\coh^2(\Psi'_{0,w'} ,A(L_{\tilde w}))=0$$ 
at every place $w'$ of ${L'_0}$ \cite[Lemma 2.1.9]{tan24}
and hence by the Hochschild-Serre spectral sequence the restriction map
$$res^L_{{L'_0}}:\bigoplus_{w'} \coh^1({L'_{0,w'}},A)\longrightarrow (\bigoplus_{\tilde w} \coh^1(L_{\tilde w},A))^{\Psi'_0}$$
is surjective. Here $w'$ is taken over all places of ${L'_0}$ and $\tilde w$ over all places of $L$.
Also, since $A_{p^\infty}(L)$ is finite and $\Psi'_0\simeq\Z_p$ is of cohomological dimension $1$, 
$$\coh^2(\Psi'_0,A_{p^\infty}(L))=0,$$ 
and hence the restriction map
$$\mathrm{res}_{L/{L'_0}}:\coh^1({L'_0}, A_{p^\infty})\longrightarrow \coh^1(L,A_{p^\infty})^{\Psi'_0}$$
is surjective. These fit into the commutative diagram of exact sequences
\begin{equation}\label{e:LL'}
\xymatrix{\coh^1(\Psi'_0,A_{p^\infty}(L))\ar@{^{(}->}[r] \ar[d] &\coh^1({L'_0}, A_{p^\infty}) \ar[d]^-{\mathcal L_{{L'_0}}} 
\ar@{->>}[r]^-{\mathrm{res}_{L/{L'_0}}}& \coh^1(L,A_{p^\infty})^{\Psi'_0} \ar[d]^-{\mathcal L_{L}^{\Psi'_0}}\\
 \bigoplus_{w'} \coh^1(\Psi'_{0,w'}, A(L_{\tilde w})) \ar@{^{(}->}[r] & \bigoplus_{w'} \coh^1({L'_{0,w'}},A)[p^\infty] \ar@{->>}[r]^-{res^L_{{L'_0}}}  & (\bigoplus_{\tilde w} \coh^1(L_{\tilde w},A))^{\Psi'_0}[p^\infty].}
\end{equation}

By \cite[Proposition 4.2]{tan13}, $\coker(\mathcal L_{{L'_0}})=0$ (The proposition is stated under the hypothesis that the places ramifying in $L/K$ are either of good ordinary or split multiplicative reduction. However, the proof carries over identically when a ramified place has non-split multiplicative reduction). Consequently, 
\begin{equation}\label{e:psi'1}
\coker(\mathcal L_{L}^{\Psi'_0})=0.
\end{equation}

The snake lemma applied to the diagram \eqref{e:LL'} yields the exact sequences of $\Gamma'$-modules
\begin{equation}\label{e:snake}
\xymatrix{0\ar[r] & F \ar[r] & \Sel_{p^{\infty}}(A/{L'_0}) \ar[r] & \Sel_{p^\infty}(A/L)^{\Psi'_0} \ar[r] & E\ar[r] & 0,}
\end{equation} 
 and
\begin{equation}\label{e:kercoker}
\xymatrix{0\ar[r] & F \ar[r] & \coh^1(\Psi'_0, A_{p^\infty}(L)) \ar[r] & \bigoplus_{w'} \coh^1(\Psi'_{0,w'}, A(L_{\tilde w}))[p^\infty] \ar[r] & E \ar[r] & 0.}
\end{equation}  
Since $A_{p^\infty}(L)$ is finite, $F$ is finite, and $E$ is a quotient of 
$\bigoplus_{w'} \coh^1(\Psi_{0,w'}, A(L_{\tilde w}))$ by a finite group. 
Put
\begin{equation}\label{e:psin}
\psi'_0=\psi_0^{p^{n}},
\end{equation}
a topological generator of $\Psi'_0$, {where \(p^n\) is the order of \(G\).}
Taking the Pontryagin duals of both sequences and then taking the tensors of all items with $\Q_\chi$ over 
$\Z_p$,
one obtains the  exact sequence  of ${\Gamma'_0}$-modules
\begin{equation}\label{e:wll'}\xymatrix{0\ar[r] & \Q_\chi\cdot\prod_{v}\mathcal W_{v}^1 \ar[r] & \Q_\chi X_L/(\psi'_0-1) \ar[r] & \Q_\chi X_{{L'_0}}\ar[r] & 0.}
\end{equation}
Here 
$\mathcal W_{v}^1$ is the direct product of
$\mathcal W_{w'}^1$ over all places $w'$ of ${L'_0}$ sitting over $v$, where 
$$\mathcal W_{w'}^1:=\coh^1(\Psi'_{0,w'}, A(L_{\tilde w}))^\vee$$ 
viewed as a $\Lambda_{{\Gamma'_{0,w'}}}$-module. There are finitely many $w'$ sitting over $v$, so $\mathcal W_v^1$ is a finite direct sum.


\subsubsection{The local factors}\label{ss:local}
Let $v$, $v'$, $w$, $w'$ and $\tilde w$ be as in \S\ref{ss:sp}. 
To investigate the structure of $\mathcal W_{v}^1$ as a ${\Lambda_{\Gamma'_0}}$-module, we observe that as ${\Gamma'_0}$-modules, 
\begin{equation}\label{e:tensor}
\mathcal W_{v}^1=\Lambda_{\Gamma'_0}\otimes_{\Lambda_{\Gamma'_{0,w'}}} \mathcal W_{w'}^1.
\end{equation}

\begin{lemma}\label{l:diamondsuit=1} If $v\not\in S$, then $\mathcal W_{v}^1=0$, and hence
$$\mathrm{CH}_{\Q_\chi\Lambda_{\Gamma_0}}(\Q_\chi\mathcal W^1_{v}/(\bar\psi_0-\chi(\bar\psi_0)))=(1).$$
\end{lemma}
\begin{proof} If $v\not\in S$, then $v$ is unramified over $L/K$. Since $L_{0,w}/K_v$ is the maximal unramified $p$-extension, $L_{\tilde w}=L_{0,w}={L'_{0,w'}}$. Hence $\Psi'_{0,w'}=0$ and $\coh^1(\Psi'_{0,w'}, A(L_{\tilde w}))=0$.

\end{proof}

\begin{lemma}\label{l:diamondsuit} For $v\in S$,
$$\mathrm{CH}_{\Q_\chi\Lambda_{\Gamma_0}}(\Q_\chi\mathcal W^1_{v}/(\bar\psi_0-\chi(\bar\psi_0)))=(\diamondsuit_{v,\chi}).$$
\end{lemma}
\begin{proof}
Let $K'_n$ denote the $n$-{\em{th}} layer of ${L'_0}/K'$.
Set $L_{1,n}=L_1 K'_n$, $\Psi'_{0,n}:=\Gal(L_{1,n}/K'_n)$ and $G_n:=\Gal(K'_n/K'_n\cap L_0)$. 
Let $v'_n$, $\tilde v_n$ be the places of $K'_n$, $L_{1,n}$ sitting below $\tilde w$. Under this setting, 
\begin{equation}\label{e:liminj}
\coh^1(\Psi'_{0,w'}, A(L_{\tilde w}))=\varinjlim_{n} \coh^1(\Psi'_{0,n,v'_n}, A(L_{1,n,\tilde v_n})),
\end{equation}
where the action of the decomposition subgroup $G_{\infty,w'}$ on the left-hand side is compatible with that of $G_{n,v'}$ on $\coh^1(\Psi'_{0,n,v'_n}, A(L_{1,n,\tilde v_n}))$.

We first consider the case in which $v$ is a good ordinary place and let
$\bar A$ denote the reduction of $A$. By \cite[(3) and Theorem 2]{tan10},
we have an exact sequence
$$\xymatrix{0\ar[r] & \bar A_{p^\infty}(\F_{v'_n})^\vee \ar[r] & \coh^1(\Psi'_{0,n,v'_n}, A(L_{1,n,\tilde v_n})) \ar[r] & \Hom(\Psi^{'0}_{0,n,v'_n},\bar A(\F_{v'_n}))\ar[r] & 0,}$$
where $\Psi^{'0}_{0,n,v'_n}\subset \Psi'_{0,n,v'_n}$ denotes the inertia subgroup.
Since ${L'_0}/K'$ is unramified at $v'$, the restriction of Galois action gives rise to an isomorphism on the inertia subgroups $\xymatrix{\Psi^{'0}_{0,w'}  \ar[r]^-{\sim} &\Psi^{'0}_{0,n,v'_n}}$. Taking the direct limit of the above sequence over $n$ and then taking the Pontryagin dual, one obtains 
\begin{equation}\label{e:wgood}
\xymatrix{0\ar[r] & \mathfrak A \ar[r] & \mathcal W_{w'}^1 \ar[r] & \mathfrak B \ar[r] & 0,}
\end{equation}
where (because $\Psi^{'0}_{0,w'}\simeq\Z_p$) 
$$\mathfrak A=\Hom(\Psi^{'0}_{0,w'}, \bar A(\F_{w'}))^\vee=\Hom(\Psi^{'0}_{0,w'}, \bar A_{p^\infty}(\F_{w'}))^\vee\simeq \bar A_{p^\infty}(\F_{w'})^\vee$$ 
and 
$$\mathfrak B=\varprojlim_n \bar A_{p^\infty}(\F_{v'_n})$$
(Here, the projective limit is taken over norm maps). Since the inertia subgroup $G_{w'}^0\subset G$ acts trivially on $\F_{w'}$,
both $\mathfrak A$ and $\mathfrak B$ are fixed by $G_{w'}^0$, so \eqref{e:wgood} induces the long exact sequence
\begin{equation*}
\xymatrix{0\ar[r] & \Q_\chi\mathfrak A \ar[r] & \Q_\chi(\mathcal W_{w'}^1)^{G_{w'}^0} \ar[r] & \Q_\chi\mathfrak B \ar[r] & \Q_\chi\Hom(G_{w'}^0, \mathfrak A)\ar[r] & \dots.}
\end{equation*}
Since $G^0_{w'}$ is finite, $\Hom(G_{w'}^0, \mathfrak A)$ is $\Z_p$-torsion, so $ \Q_\chi\Hom(G_{w'}^0, \mathfrak A)=0$. Hence $\Q_\chi\mathcal W_{w'}^1$ is fixed by $G^0_{w'}$ and the action of $G$ on $\Q_\chi\mathcal W_{w'}^1$ factors through $G/G^0_{w'}$.
 Since $\chi$ generates $\hat G$, if ${L'_0}/K$ is ramified at $v$, then $G^0_{w'}\not=0$, so the module $\Q_\chi\mathcal W_v^1$ has trivial $\chi$-eigenspace, whence $\Q_\chi\mathcal W_v^1/(\bar\psi_0-\chi(\psi_0))=0$ (the $\chi$-eigenspace and the $\chi$-isotypic quotient are isomorphic).

If ${L'_0}/K$ is unramified at $v$, then ${L'_{0,w'}}=L_{0,w}$, and hence \eqref{e:wgood} is nothing but \cite[(37)]{tan13}. The decomposition subgroup ${\Gamma'_{0,w'}=\Gamma'_{0,v}=\Z_p}$ and the natural map
$\Gamma_{0,v'}\longrightarrow \Gamma_{0,v}$ is an isomorphism, sending $[v]_{{L'_0}/K}$ to $[v]_{L_0/K}$. 
We choose $\psi_0$ such that $[v]_{{L'_0}/K}=\bar\psi_0^a\cdot\varphi$, where $a=p^{m}$, $m\leq n$, and $\varphi\in \Psi'_1$. Then 
$[v']_{{L'_0}/K'}=\varphi^c=[v]_{{L'_0}/K}^c$, $c=p^{n-m}$. Here $p^n$ is the order of $G$ (see \eqref{e:psin}).


By \cite[Proposition 3.6]{tan13}, as a $\Gamma'_{0,v'}$-module\footnote{The Proposition only gives pseudo-isomorphism over $\Lambda_{\Gamma_v}$, but since ``pseudo-null" means ``finite", by extending the coefficients
to the field $\Q_\chi$, we get an isomorphism.}
$$
\Q_\chi \mathcal W_w^1=\Q_\chi \Lambda_{{\Gamma'_{0,v}}}/(1-\alpha_{v}^{-1}[v]_{{L'_0}/K}^{-1})\oplus \Q_\chi \Lambda_{{\Gamma'_{0,v}}}/(1-\alpha_{v}^{-1}[v]_{{L'_0}/K}),
$$ 
so as a ${\Gamma_0'}$-module,
$$\Q_\chi \mathcal W_v^1=\Q_\chi \Gamma'\otimes_{\Lambda_{\Gamma'_{v}}} \mathcal W_w^1= \Q_\chi \Lambda_{\Gamma'}/(1-\alpha_{v}^{-1}\bar\psi_0^{-a}\varphi^{-1})\oplus \Q_\chi \Lambda_{\Gamma'}/(1-\alpha_{v}^{-1}\bar\psi_0^a\varphi).$$
Since $\chi(\varphi)=1$, $\chi(\bar\psi_0^a)=\chi([v]_{{L'_0}/K})=\chi([v]_{K'/K})$. Also, $p^{{L'_0}}_{L_0}(\varphi)=[v]_{L_0/K}$. 
Thus, as a $\Gamma_0$-module, $\Q_\chi \mathcal W_v^1/(\bar\psi_0-\chi(\bar\psi_0))$ can be expressed as
\begin{equation*}
\Q_\chi \Lambda_{{\Gamma'_0}}/(1-\alpha_{v}^{-1}\chi([v]_{K'/K})^{-1}[v]_{L_0/K}^{-1})\oplus \Q_\chi \Lambda_{{\Gamma'_0}}/(1-\alpha_{v}^{-1}\chi([v]_{K'/K})[v]_{L_0/K}).
\end{equation*}
This shows that
\begin{equation}\label{e:ordin}\begin{array}{rcl}
(\diamondsuit_{v,\chi})&=&(1-\alpha_{v}^{-1}\chi([v]_{K'/K})^{-1}[v]_{L_0/K}^{-1})(1-\alpha_{v}^{-1}\chi([v]_{K'/K})[v]_{L_0/K})\\
{}&=& \mathrm{CH}_{\Q_\chi\Lambda_{\Gamma_0}}(\Q_\chi\mathcal W_v^1/(\bar\psi_0-\chi(\bar\psi_0))).
\end{array}
\end{equation}

Next, we consider the case of split multiplicative $v$. Let $Q$ denote the local Tate period.
The exact sequence
$$\xymatrix{0\ar[r] & Q^\Z \ar[r] &  L_{\tilde w}^* \ar[r] & A(L_{\tilde w}) \ar[r] & 0}$$
induces
$$\xymatrix{\cdots \ar[r] & \coh^1(\Psi'_{0,w'}, L_{\tilde w}^*)\ar[r]  & \coh^1(\Psi'_{0,w'}, A(L_{\tilde w})) \ar[r] &
\coh^2(\Psi'_{0,w'}, Q^\Z)\ar[r] & \cdots.}$$
By Hilbert Theorem 90, $\coh^1(\Psi'_{0,w'}, L_{\tilde w}^*)=0$. Since $Q^\Z\simeq \Z$ and $\Gamma$ is commutative,  $\coh^2(\Psi'_{0,w'}, Q^\Z)$ is fixed by $G_{w'}^0$ and so is $\coh^1(\Psi'_{0,w'}, A(L_{\tilde w}))$. Therefore, if $G_{w'}^0\not=0$, then $\Q_\chi\mathcal W_v^1/(\bar\psi_0-\chi(\bar\psi_0))=0$.

If ${L'_0}/L_0$ is unramified at $w$, 
then by \cite[Proposition 3.7(c)]{tan13},
as a ${\Gamma'_{0,v}}$-module, 
$$\Q_\chi \mathcal W_w^1/(\bar\psi_0-\chi(\bar\psi_0))=\Q_\chi \Lambda_{{\Gamma'_{0,v}}}/(1-[v]_{{L'_0}/K}).$$
Similarly, it follows that, under the above notation, 
\begin{equation}\label{e:spm}\begin{array}{rcl}
\mathrm{CH}_{\Q_\chi\Lambda_{\Gamma_0}}(\Q_\chi\mathcal W_v^1/(\bar\psi_0-\chi(\bar\psi_0)))&=&
(1-\chi([v]_{K'/K})[v]_{L_0/K})\\
&=& (-1+\chi([v]_{K'/K})^{-1}[v]_{L_0/K}^{-1})\\
&=& (\diamondsuit_{v,\chi}).
\end{array}
\end{equation}

If $v\in S$ is a non-split multiplicative place, in our situation, by \cite[Lemma 2.2.4]{tan24}

$$\Q_\chi \mathcal W_w^1/(\bar\psi_0-\chi(\bar\psi_0))=
\begin{cases}
0, &\text{ if } p\not=2;\\
\Q_p\Lambda_{{\Gamma'_{0,v}}}/(-1-[v]_{{L'_0}/K}), &\text{ if } p=2.
\end{cases}
$$
It also follows that
\begin{equation}\label{e:nonspm}
\mathrm{CH}_{\Q_\chi\Lambda_{\Gamma_0}}(\Q_\chi\mathcal W_v^1/(\bar\psi_0-\chi(\bar\psi_0)))=(\diamondsuit_{v,\chi}).
\end{equation}
\end{proof}

Taking the $\chi$-isotypic quotients of items of \eqref{e:wll'}, applying Lemma \ref{l:diamondsuit},
and using the fact that $(\Q_\chi X_L/(1-\psi_0^{p^{n_\chi}}))/(\bar\psi_0-\chi(\bar\psi_0))=\Q_\chi X_L/(\bar\psi_0-\chi(\bar\psi_0))$, we obtain
$$
\mathrm{CH}_{\Q_\chi\Lambda_{\Gamma_0}}(\Q_\chi X_L/(\psi_0-\chi(\psi_0))) =(\prod_{v\in S}\diamondsuit_{v,\chi})\cdot \mathrm{CH}_{\Q_\chi\Lambda_{\Gamma_0}}(\Q_\chi X_{{L'_0}}/(\bar\psi_0-\chi(\bar\psi_0))).
$$
Then Lemma \ref{l:xlpsio} follows from \eqref{e:(chi)}.

\subsection{The proof of Theorem \ref{t:1st}}\label{su:pf1st}
In \S\ref{ss:pfagchi}, \S\ref{ss:pfanchi}, we shall show that in $\Q_\chi\Lambda_{\Gamma_0}$
the ideals
\begin{equation}\label{e:agchi}
^\chi\mathrm{CH}_{\Q_\chi\Lambda_{\Gamma_0}}(X_L)=\prod_{v\in S}(\diamondsuit_{v,\chi})\cdot (\mathsf c_\chi),\end{equation}

\begin{equation}\label{e:anchi}
(^\chi\mathscr L_{A/L})=\prod_{v\in S}(\diamondsuit_{v,\chi})\cdot (\mathsf c_\chi).
\end{equation}
Since $\diamondsuit_{v,\chi}\not=0$, for all $v$, the above equalities imply that as ideals in $\O_\chi\Lambda_{\Gamma_0}$, 
$$ ^\chi\mathrm{CH}_{\Lambda_{\Gamma_0}}(X_L)\cdot a=(^\chi\mathscr L_{A/L})\cdot b,\quad\;\text{for some non-zero}\; a,b\in\O_\chi.$$
Lemma \ref{l:mu} below implies that we can take $a=b=1$, which proves Theorem \ref{t:1st}. 

The $\mu$-invariant $\mu(f)$
of an $f\in\Q_\chi\Lambda_{\Gamma_0}$ is defined to be the rational number $\ord_p m$, where $m\in\Q_\chi$ such that $m^{-1}\cdot f\in \O_\chi\Lambda_{\Gamma_0}$ but not divisible by the prime element $\pi_\chi$ of $\O_\chi$.
Here $\ord_p$ is the valuation on $\bar\Q_p$ with $\ord_p p=1$. For a principal ideal $I=(\eta)$ in $\O_\chi\Lambda_{\Gamma_0}$, put $\mu(I):=\mu(\eta)$.

\begin{lemma}\label{l:value}For an $f\in \O_\chi\Lambda_{\Gamma_0}$ having $\mu(f)=0$, the set
$$\Sigma_f:=\{\omega\in\hat\Gamma_0\;\mid\; \ord_p \omega(f) \geq \ord_p \pi_\chi\}$$ 
is finite.
\end{lemma}
\begin{proof}The lemma is actually a direct consequence of \cite[Theorem 2.3]{monsky}, because 
$\Nm(f):=\prod_{\sigma\in \Gal(\Q_\chi/\Q_p)}{}^\sigma f\in\Lambda_{\Gamma_0}$ has trivial $\mu$-invariant and 
$$\Sigma_f=\{\omega\in\hat\Gamma_0\;\mid\; \ord_p \omega(\Nm(f)) \geq 1\}.$$

\end{proof}

\begin{lemma}\label{l:mu}We have $\mu(^\chi\mathrm{CH}_{\Lambda_{\Gamma_0}}(X_L))=\mu(^\chi\mathscr L_{A/L})$.
\end{lemma}
\begin{proof}For every $\omega\in\hat\Gamma_0$, we have 
$$\omega(^\chi\mathrm{CH}_{\Lambda_{\Gamma_0}}(X_L))=\omega\chi(\mathrm{CH}_{\Lambda_{\Gamma_0}}(X_L)),\quad 
\omega(^\chi\mathscr L_{A/L})=\omega\chi(\mathscr L_{A/L}).$$ 

Since $L_0\subset L$, $\dag_{A/L}=1$, so by \cite[Proposition 1.1.1]{tan24}, for all $\omega\in\hat\Gamma_0$,
\begin{equation}\label{e:ordp}
\ord_p \omega({}^\chi\mathrm{CH}_{\Lambda_{\Gamma_0}}(X_L))=\ord_p \omega(^\chi\mathscr L_{A/L}).
\end{equation}
Suppose $m^{-1}\cdot {}^\chi \mathscr L_{A/L}\in\O_\chi\Lambda_{\Gamma_0}$ but not divisible by $\pi_\chi$.

If $\mu({}^\chi \mathscr L_{A/L})< \mu ({}^\chi\mathrm{CH}_{\Lambda_{\Gamma_0}}(X_L))$,
then $m^{-1}\cdot {}^\chi\mathrm{CH}_{\Lambda_{\Gamma_0}}(X_L)\in\O_\chi\Lambda_{\Gamma_0}$ and is divisible by $\pi_\chi$, thus by \eqref{e:ordp}, for all $\omega\in\hat\Gamma_0$,
$\ord_p \omega(m^{-1}\cdot {}^\chi \mathscr L_{A/L})\geq\ord_p\pi_\chi$. This would contradict Lemma \ref{l:value}. This shows  $\mu({}^\chi \mathscr L_{A/L})\geq \mu ({}^\chi\mathrm{CH}_{\Lambda_{\Gamma_0}}(X_L))$. The inequality in the other direction can be proved by a similar argument.
\end{proof}

\subsubsection{The ideal $^\chi\mathrm{CH}_{\Q_\chi\Lambda_{\Gamma_0}}(X_L)$}\label{ss:pfagchi}
There is an exact sequence of $\Lambda_\Gamma$-modules
\begin{equation}\label{e:stand}
\xymatrix{0 \ar[r] & [X_L] \ar[r] & X_L \ar[r] & N \ar[r] & 0,}
\end{equation}
in which $N$ is pseudo-null and
$$[X_L]=\bigoplus_i \Lambda_\Gamma/(f_i^{n_i})\oplus\bigoplus_j \Lambda_\Gamma/(p^{m_j}).$$
Here, each $f_i$ is irreducible in $\Lambda_\Gamma$, not divisible by $p$. Apply $\Q_p\otimes_{\Z_p}$
to the terms of \eqref{e:stand} and obtain the exact sequence of $\Q_\chi\Lambda_\Gamma$-modules:
\begin{equation}\label{e:Qstand}
\xymatrix{0 \ar[r] & \Q_\chi [X_L] \ar[r] & \Q_\chi X_L \ar[r] & \Q_\chi N \ar[r] & 0,}
\end{equation}
where $\Q_\chi N$ is finite $\Q_\chi$-dimensional (\cite{gre78}, paragraph before Lemma 3)
and
$$\Q_\chi\cdot [X_L]=\bigoplus_i \Q_\chi\Lambda_\Gamma/(f_i^{n_i}).$$

Multiplying the sequence \eqref{e:Qstand} with $f:=\psi_0-\chi(\psi_0)$ and using the snake lemma, we obtain the exact sequence
\begin{equation}\label{e:psisnake}
\xymatrix{ \Q_\chi [X_L][f] \ar@{^(->}[r] & \Q_\chi X_L[f] \ar[r] & 
\Q_\chi N[f] \ar[lld] \\
\Q_\chi [X_L]/f \ar[r] & \Q_\chi X_L/f \ar@{->>}[r] &  \Q_\chi N/f.}
\end{equation}

The module $\Q_\chi X_L[f]\subset X_L^{\Psi'_0}$ and 
$$X_L^{\Psi'_0}=(\Sel_{p^\infty}(A/L)/(\psi'_0-1)\Sel_{p^\infty}(A/L))^\vee.$$
Since by \eqref{e:psi'1} and
\cite[(49)]{tan13}, $\Sel_{p^\infty}(A/L)/(\psi'_0-1)\Sel_{p^\infty}(A/L)=0$, so $\Q_\chi X_L[f]=0$ and the above \eqref{e:psisnake} also says $\Q_\chi [X_L][f]=0$. 

Take the characteristic ideals of the terms in the exact sequence
$$\xymatrix{0\ar[r]  & \Q_{\chi}N[f]\ar[r] &  \Q_{\chi}N\ar[r]^-{\cdot f} &  \Q_{\chi}N\ar[r] &  \Q_{\chi}N/f\ar[r] & 0}$$
use the fact that $\Q_\chi N$ is finite $\Q_\chi$-dimensional (whence $\mathrm{CH}_{\Q_\chi\Lambda_{\Gamma_0}}(\Q_\chi N)\not=0$), and obtain  
$$\mathrm{CH}_{\Q_\chi\Lambda_{\Gamma_0}}(\Q_\chi N[f])=\mathrm{CH}_{\Q_\chi\Lambda_{\Gamma_0}}(\Q_\chi N/f).$$ 
Thus,
the exact sequence \eqref{e:psisnake} implies
$$\mathrm{CH}_{\Q_\chi\Lambda_{\Gamma_0}}(\Q_\chi X_L/(\psi_0-\chi(\psi_0)))=\mathrm{CH}_{\Q_\chi\Lambda_{\Gamma_0}}(\Q_\chi [X_L]/(\psi_0-\chi(\psi_0)))={}^\chi\mathrm{CH}_{\Q_\chi\Lambda_{\Gamma_0}}(\Q_\chi X_L).$$
So \eqref{e:agchi} follows from Lemma \ref{l:xlpsio}.

\subsubsection{The element $^\chi(\mathscr L_{A/L}^\sharp)$}\label{ss:pfanchi}
In view of \eqref{e:anfe}, in \eqref{e:anchi},
we can replace $\mathscr L_{A/L}$ by $\mathscr L_{A/L}^\sharp$.

For $\omega\in\hat\Gamma_0\subset \hat\Gamma$,
$$\omega(^\chi(\mathscr L_{A/L}^\sharp))=\omega\chi(\mathscr L_{A/L}^\sharp)=\omega^{-1}\chi^{-1}(\mathscr L_{A/L}),$$
thus by \eqref{e:intp} and \eqref{e:fudge}, 
\begin{equation}\label{e:omega}
\omega(^\chi(\mathscr L_{A/L}^\sharp))=\alpha_{D_{\omega^{-1}\chi^{-1}}}\cdot \tau_{\omega^{-1}\chi^{-1}}\cdot q^{\frac{\Delta}{12}+\kappa-1}\cdot \Xi_{S,\omega^{-1}\chi^{-1}}\cdot L_A(\omega^{-1}\chi^{-1},1).
\end{equation}
Here we use the fact that $\dag_{A/L}=1$ in our setting. Because $D_{\omega^{-1}\chi^{-1}}=D_{\chi^{-1}}=D_\chi$,
$$\alpha_{D_{\omega^{-1}\chi^{-1}}}=\alpha_{D_\chi}$$
a constant in $\Q_\chi$ independent of $\omega$.
Next, by \cite[(16)]{tan24}, the Gauss sum 
$$\tau_{\omega^{-1}\chi^{-1}}=\tau_{\omega^{-1}} \cdot \tau_{\chi^{-1}}=\omega(\mathsf
 {F}_q^{-2+2\kappa+\deg (D_\chi)})\cdot \tau_{\chi^{-1}}.$$

Since 
$$\omega\chi([v]_{{L_0'}/K})=\omega([v]_{L_0/K})\cdot\chi([v]_{K'/K}),$$
we find that
$$\Xi_{S,\omega^{-1}\chi^{-1}}=\prod_{v\in S} \omega(\diamondsuit_v)\cdot \omega(\prod_{v\in S_{ns}} \epsilon_v),$$
where $S_{ns}\subset S_m$ is the subset consisting of non-split multiplicative places and 
$$\epsilon_v:=
\begin{cases}
-1-\chi([v]_{K'/K})\cdot [v]_{L_0/K}, \text{ if } p\not=2;\\
1, \text{ if } p=2.
\end{cases}
$$

Since $\mathsf F_q$ and $\epsilon_v$ are units in $\Lambda_{\Gamma_0}$,
the element
$$\dag_\chi:=\alpha_{D_{\chi}}\cdot \tau_{\chi^{-1}} \cdot q^{\frac{\deg(\Delta)}{12}+\kappa-1}\cdot \mathsf
 {F}_q^{-2+2\kappa+\deg (D_\chi)}\cdot \prod_{v\in S_{ns}} \epsilon_v$$ 
 is a unit in
 $\Q_\chi\Lambda_{\Gamma_0}$.  
 So \eqref{e:omega} says
$$^\chi (\mathscr L_{A/L}^\sharp)=\dag_\chi\cdot\prod_{v\in S} \diamondsuit_{v,\chi}\cdot\mathsf c_\chi.$$
Thus, \eqref{e:anchi} follows.

\section{The Iwasawa Main Conjecture}\label{s:imc}
Now we prove the Iwasawa Main Conjecture. The $d=0$ case is already proven.
\subsection{Preliminary}\label{su:prelim}
The following material will be useful in the ongoing discussion. The first lemma is extracted from \cite[Corollary 5.2.2]{tan24}.  

\begin{lemma}\label{l:5.2.2}For an intermediate $\Z_p^e$-extension $L'/K$ of $L/K$,
$p^L_{L'}(\dag_{A/L})\cdot\vartheta_{L/L'}\not=0$ if and only if $\Gamma_v'\not=0$, for all split-multiplicative $v\in S$.
\end{lemma}

\begin{corollary}\label{c:5.2.2}
Suppose $K_\infty^{(p)}\subset L'$, or $L=K_\infty^{(p)}L'$. Then 
\begin{equation}\label{e:ptheta}
p^L_{L'}(\dag_{A/L})\cdot\vartheta_{L/L'}\not=0,
\end{equation}
 hence, the Iwasawa Main Conjecture holds for $L'/K$, if and only if 
\begin{equation}\label{e:pll'}
(p^L_{L'}(\mathscr L_{A/L}))=p^L_{L'}(\mathrm{CH}_{\Lambda_\Gamma}(X_L)).
\end{equation}
\end{corollary}


\begin{proof}
If $K_\infty^{(p)} \subset L'$, then $\Gamma'_v \neq 0$ for every place $v$. If $L = K_\infty^{(p)}L'$, then the ramification locus of $L'/K$ is identical to $S$, ensuring that for $v \in S$, $\Gamma'_v$ is non-trivial. Hence, by Lemma \ref{l:5.2.2}, the term $p_{L'}^L(\dag_{A/L}) \cdot \vartheta_{L/L'}$ is non-zero. By the specialization formulae \eqref{arithspf} and \eqref{anspf}, the identity (\ref{e:pll'}) is then equivalent to the Iwasawa Main Conjecture for $A/K$ over $L'/K$.
\end{proof}

\begin{lemma}\label{l:preparation}If $f\in\Z_p[[t_0,...,t_n]]$ satisfies $\mu(f(t_0,0,...,0))=\mu(f)$, then there is some non-negative integer $\lambda$ and some $u\in\Z_p[[t_0,...,t_n]]^*$ such that
$$p^{-\mu(f)}\cdot f=u\cdot (t_0^\lambda+\sum_{i=0}^{\lambda-1} a_i\cdot t_0^i),\;a_i\in (p,t_1,...,t_n)\subset
\Z_p[[t_1,...,t_n]],\; i=0,...,\lambda-1.$$
\end{lemma}
\begin{proof}This is from the Weierstrass Preparation Theorem
\cite[VII, \S 8, Proposition 6]{bou}, because $p^{-\mu(f)}\cdot f(t_0,0,...,0)$ does not belong to $p\cdot \Z_p[[t_0]]$, so $p^{-\mu(f)}\cdot f$, as a power series in $t_0$ over $\Z_p[[t_1,...,t_n]]$,
 has some coefficient belonging to $\Z_p[[t_1,...,t_n]]^*$.
\end{proof}

\subsection{The $\mathscr P$-roots of two variable formal power series}\label{su:roots}
Let $\mathscr O$ denote the ring of integers in $\bar\Q_p$ and let $\mathscr P\subset \mathscr O$ be the maximal ideal. If $v_p$ denotes the valuation on $\bar\Q_p$ such that $v_p(p)=1$, then 
$\mathscr P$ consists of $x$ having $v_p(x)>0$.

Let $\O$ be the ring of integers in a finite extension of $\Q_p$. We call a point $(a,b)\in\mathscr P\times \mathscr P$ a $\mathscr P$-root of $f\in\O[[t_0,t_1]]$, if $f(a,b)=0$. 

\begin{lemma}\label{l:intpts}Two formal power series $f,g\in\O[[t_0,t_1]]$ having infinitely many $\mathscr P$-roots in common have a non-unit common divisor in $\O[[t_0,t_1]]$.
\end{lemma}
\begin{proof} See the proof of  \cite[Lemma 6.6.2]{tan24}.
\end{proof}



\begin{lemma}\label{l:root}Let $f,g\in\Z_p[[t_0,t_1]]$ be such that $\mu(f)=\mu(f(t_0,0))=\mu(g)=\mu(g(t_0,0))$. If for infinitely many $\zeta\in\Bmu_{p^\infty}$,
in $\Z_p[\zeta][[t_0]]$ the ideals 
\begin{equation}\label{e:fgzeta}
(f(t_0,\zeta-1))=(g(t_0,\zeta-1)),
\end{equation}
then in $\Z_p[[t_0,t_1]]$ the ideals
$$(f)=(g).$$
\end{lemma}




\begin{proof} Let $Z\subset \Bmu_{p^\infty}$ be the set of $\zeta$ satisfying \eqref{e:fgzeta}. Let $h = \text{g.c.d.}(f, g)$ and write $f = h \cdot f_1$ and $g = h \cdot g_1$ where $f_1, g_1$ are relatively prime. First, we show that the hypothesis on $\mu$-invariants holds for $f_1$ and $g_1$.
Indeed, since $\mu$ is additive for products in $\Z_p[[t_0,t_1]]$, we have
\[
\mu(f)=\mu(h)+\mu(f_1),\qquad \mu(g)=\mu(h)+\mu(g_1).
\]
After specializing $t_1=0$, we likewise have
\[
\mu(f(t_0,0))=\mu(h(t_0,0))+\mu(f_1(t_0,0)),\qquad
\mu(g(t_0,0))=\mu(h(t_0,0))+\mu(g_1(t_0,0)).
\]
Using the assumptions $\mu(f)=\mu(f(t_0,0))$ and $\mu(g)=\mu(g(t_0,0))$, we get
\begin{align}
\mu(h)+\mu(f_1) &= \mu(h(t_0,0))+\mu(f_1(t_0,0)), \label{e:mu1}\\
\mu(h)+\mu(g_1) &= \mu(h(t_0,0))+\mu(g_1(t_0,0)). \label{e:mu2}
\end{align}
On the other hand, specialization cannot increase the $p$-divisibility, hence
\[
\mu(h(t_0,0))\ge \mu(h),\qquad \mu(f_1(t_0,0))\ge \mu(f_1),\qquad \mu(g_1(t_0,0))\ge \mu(g_1).
\]
Comparing with \eqref{e:mu1} forces equality in both inequalities for $h$ and $f_1$, i.e.
\[
\mu(h(t_0,0))=\mu(h)\quad\text{and}\quad \mu(f_1(t_0,0))=\mu(f_1).
\]
Similarly, comparing with \eqref{e:mu2} gives
\[
\mu(g_1(t_0,0))=\mu(g_1).
\]
Hence, $f_1,g_1$ satisfy the hypothesis of Lemma \ref{l:preparation}, and we can write $f_1 = u_1 \cdot P_{f_1}(t_0)$ and $g_1 = u_2 \cdot P_{g_1}(t_0)$, where $u_1, u_2$ are units and $P_{f_1}, P_{g_1}$ are distinguished polynomials in $t_0$ of degrees $m$ and $n$, respectively. Note that $h(t_0,\zeta-1)\neq 0$ for all but finitely many $\zeta$. Indeed, write
$$
h(t_0,t_1)=\sum_{i\ge 0} a_i(t_1)\,t_0^i\qquad (a_i(t_1)\in\Z_p[[t_1]]).
$$
Since $h\neq 0$, there exists $i$ with $a_i(t_1)\neq 0$.
A nonzero one-variable power series $a_i(t_1)\in\Z_p[[t_1]]$ cannot vanish at
infinitely many distinct points $t_1=\zeta-1$ with $\zeta\in\Bmu_{p^\infty}$
(because $\zeta-1\to 0$ and a nonzero $p$-adic analytic function has only finitely
many zeros in a small disk unless it is identically $0$).
Hence $a_i(\zeta-1)\neq 0$ for all but finitely many $\zeta$, and therefore
$h(t_0,\zeta-1)\neq 0$ in $\Z_p[\zeta][[t_0]]$ for all but finitely many $\zeta$.

Let $Z_0\subset Z$ be an infinite subset such that $h(t_0,\zeta-1)\neq 0$ for all $\zeta\in Z_0$.

Assume $m > 0$. Then for every $\zeta$ in $Z_0$, the specialization $f_1(t_0, \zeta-1)$ has $m$ roots (counted with multiplicity) in $\mathscr{P}$. Since $(f(t_0, \zeta-1)) = (g(t_0, \zeta-1))$, these are also roots of $g_1(t_0, \zeta-1)$, providing common roots of the form $(c, \zeta-1) \in \mathscr{P} \times \mathscr{P}$ for the two-variable functions $f_1$ and $g_1$. Since there are infinitely many distinct such $\zeta$, this yields an infinite set of common roots to $f_1$ and $g_1$, contradicting the hypothesis that they have no common divisors by Lemma \ref{l:intpts}. Thus, $m = 0$. By a symmetric argument, $n = 0$. It follows that $f_1, g_1 \in \mathbb{Z}_p[[t_0, t_1]]^*$, and hence $(f) = (g)$.
\end{proof}

\begin{remark}
Consider the two-variable power series
$$
f:=t_1^2+p^3\,t_0 t_1 + p^3(t_0+p), \qquad
g:=t_1^2 + p^3(t_0+p) \in \mathbb{Z}_p[[t_0,t_1]].
$$
They are relatively prime in $\mathbb{Z}_p[[t_0,t_1]].$

Let $\zeta\in\mu_{p^\infty}$ with $\zeta \neq 1$ and put $\delta=\zeta-1$.
In the ring of integers $\mathcal{O}_\zeta$ one has $(p)=(\delta^e)$ with
$e=[\mathbb{Q}_p(\zeta):\mathbb{Q}_p]$, hence $p^3\in(\delta^2).$ It follows that in
$\mathcal{O}_\zeta[[t_0]]$ both $f(t_0,\zeta-1)$ and $g(t_0,\zeta-1)$ generate the ideal
$(\delta^2),$ so
$$
\bigl(f(t_0,\zeta-1)\bigr)=\bigl(g(t_0,\zeta-1)\bigr).
$$
On the other hand, for $\zeta=1$ one has $f(t_0,0)=g(t_0,0)=p^3(t_0+p),$ but $f$ and $g$
do not generate the same ideal in $\mathbb{Z}_p[[t_0,t_1]].$ Indeed, if they did then $f=ug$ for a
unit $u\in \mathbb{Z}_p[[t_0,t_1]]^\times.$ Writing
$u=u_0(t_0)+u_1(t_0)t_1+u_2(t_0)t_1^2+\cdots$ and comparing the coefficient of $t_1$ gives
$u_1(t_0)\,p^3(t_0+p)=p^3 t_0,$ so $(t_0+p)\mid t_0,$ a contradiction.

This shows that Lemma \ref{l:root} would be false without the additional hypothesis on
$\mu$-invariants. Accordingly, that hypothesis is essential for the results that follow.
\end{remark}

\subsection{The induction}\label{su:induction}Write $L_0$ for $K_\infty^{(p)}$ and denote
$\tilde L=L_0L$. In view of Corollary \ref{c:5.2.2}, it is sufficient to show that \eqref{e:imc} holds for $\tilde L/K$, because then \eqref{e:imc} for $L/K$ follows by taking $p^{\tilde L}_L$.
Thus, without loss of generality, for proving \eqref{e:imc},
we may assume that $L=\tilde L$. 
If $d=1$, then $L=L_0$, and \eqref{e:imc} is proven in \cite{lltt16}.

For $d\geq 2$, we shall use the induction on $d$. Since $\Psi_1:=\Gal(L/L_0)$ is a maximal (in the sense that $\Gamma/\Psi_1$ is torsion free) $\Z_p$ submodule in $\Gamma$ of rank $d-1$, we can find a maximal
$\Z_p$ submodule $\Psi_0\subset\Gamma$ of rank $1$ such that $\Gamma$ is the direct sum of $\Psi_1$ and $\Psi_0$. Galois theory says that if $L_1=L^{\Psi_0}$, then $L_1$ is a $\Z_p^{d-1}$-extension of $K$ such that $L=L_0L_1$, $L_0\cap L_1=K$. Denote $\Gamma_0=\Gal(L_0/K)$, $\Gamma_1=\Gal(L_1/K)$.
Let $\sigma_0$ and $\sigma_1,...,\sigma_{d-1}$ respectively be topological generators of $\Psi_0$ and $\Psi_1$, and put $t_i=\sigma_i-1$, $i=0,...,d-1$.

Suppose $d=2$. A character $\chi\in\hat\Gamma_1\subset \hat\Gamma$ gives rise to
a $\zeta\in\Bmu_{p^\infty}$ such that $\chi(\sigma_0)=1$, $\chi(\sigma_1)=\zeta$, and vice versa.
For each $f\in \Lambda_\Gamma=\Z_p[[t_0,t_1]]$, by identifying $\Z_p[\zeta][[t_0]]$ with $\Z_p[\zeta]\Lambda_{\Gamma_0}$, we identify $f(t_0,\zeta-1)$ with ${}^\chi f$. Thus, Theorem \ref{t:1st},
the Hypothesis and 
Lemma \ref{l:root} together imply that as ideals of $\Lambda_\Gamma$,
$$(\mathscr L_{A/L})=\mathrm{CH}_{\Lambda_\Gamma}(X_L).$$

\subsection{The $d>2$ case}\label{su:d>2} 
Again, write $L_0$ for $K_\infty^{(p)}$ and write $L=L_0L_1$, $L_1$ a $\Z_p^{d-1}$-extension of $K$ such that
$L_0\cap L_1=K$.
Let $\sigma_1,...,\sigma_{d-1}$ be topological generators of $\Gal(L/L_0)$,
let $\sigma_0$ be a topological generator of $\Gal(L/L_1)$, so that  $\sigma_0,\sigma_1,...,\sigma_{d-1}$ form a $\Z_p$-basis of $\Gamma$. Put $t_i=\sigma_i-1$.

Let $\eta_L\in\Lambda_\Gamma$ denote a generator of $\mathrm{CH}_{\Lambda_\Gamma}(X_L)$.
Let $m$ denote the common $\mu$-invariant of $\mathscr L_{A/L}$ and $\eta_L$. The Hypothesis together with Lemma \ref{l:preparation} implies that
$$p^{-m}\cdot \mathscr L_{A/L}(t_0,t_1,...,t_{d-1})=u_1\cdot \xi_0,\quad
p^{-m}\cdot \eta_{L}(t_0,t_1,...,t_{d-1})=u_2\cdot \eta_0,$$
where $u_1,u_2\in \Z_p[[t_0,t_1,...,t_{d-1}]]^*$ and $\xi_0$, $\eta_0$ are distinguished polynomials in $t_0$ over $\Z_p[[t_1,...,t_{d-1}]]$, in the sense that 
$$\xi_0=t_0^{\nu_1}+\sum_{i=0}^{\nu_1-1}a_i\cdot t_0^i,\;
\eta_0=t_0^{\nu_2}+\sum_{j=0}^{\nu_2-1}b_j\cdot t_0^j,\;\;a_i,b_j\in (p,t_1,...,t_{d-1})\subset \Z_p[[t_1,...,t_{d-1}]].$$

Consider a $\Z_p^{d-1}$-extension $L'$ of $K$ such that $L_0\subset L'\subset L$. The Hypothesis also holds for $L'/K$, since for $f\in \Lambda_{\Gamma}$, one has the inequalities
\begin{equation}\label{e:inequalities}
\mu(p^L_{L_0}(f))\geq \mu(p^L_{L'}(f))\geq \mu(f).
\end{equation}
Thus, in this situation, since $L_0\subset L'$, the induction hypothesis asserts that \eqref{e:imc} holds for $L'/K$, so by Corollary \ref{c:5.2.2},
\begin{equation}\label{e:ll'imc}
(p^L_{L'}(\mathscr L_{A/L}))=(p^L_{L'}(\eta_L)).
\end{equation}

Choose a topological generator $\sigma'_1$ of $\Gal(L/L')$, extend it to a $\Z_p$-basis $\sigma'_1,...,\sigma'_{d-1}$ of $\Gal(L/L_0)$, and write $s'_j=\sigma'_j-1$. 
Then $\Z_p[[s'_1,...,s'_{d-1}]]=\Z_p[[t_1,...,t_{d-1}]]$, and hence
$\xi_0$ and $\eta_0$ are distinguished polynomials in $t_0$ over $\Z_p[[s'_1,...,s'_{d-1}]]$. 
It follows that $p^L_{L'}(\xi_0)$ and $p^L_{L'}(\eta_0)$
are distinguished polynomials in $t_0$ over $\Z_p[[s'_2,...,s'_{d-1}]]$. Thus, \eqref{e:ll'imc} together with the uniqueness of the associated distinguished polynomial in the Weierstrass Preparation Theorem implies that
as elements in $\Lambda_\Gamma$, 
$$p^L_{L'}(\xi_0) =p^L_{L'}(\eta_0).$$
This shows $\xi_0-\eta_0$ belongs to $\ker(p^L_{L'})\subset \Lambda_\Gamma$. It follows that
$\xi_0-\eta_0$ is divisible by $s'_1$, which is a generator of the ideal $\ker(p^L_{L'})$.

Suppose $L''$ is also a $\Z_p^{d-1}$-extension of $K$, with $L_0\subset L''\subset L$. Let $\sigma''_1$ be a topological generator of $\Gal(L/L'')$ and denote $s''_1+1:=\sigma''_1$. Then $\xi_0-\eta_0$ is divisible by $s''_1$.
If $L''\neq L'$, then $s'_1$ and $s''_1$ are relatively prime in $\Lambda_\Gamma$.
Indeed, via the identification
\[
\Lambda_\Gamma \simeq \Z_p[[t_0,t_1,\dots,t_{d-1}]],
\]
both $s'_1$ and $s''_1$ lie in the regular local subring
$\Z_p[[t_1,\dots,t_{d-1}]]$ and are irreducible, since their linear terms
correspond to primitive vectors in $\Z_p^{d-1}$.
If they were associates, their linear parts would be proportional,
which would imply
$\Gal(L/L')=\Gal(L/L'')$, hence $L'=L''$.
Thus for $L'\neq L''$ they are non-associate irreducibles,
and therefore relatively prime in $\Lambda_\Gamma$. The assignment 
$L'\mapsto \Q_p\otimes L/L_0$ gives a one-to-one correspondence between the set of
$\Z_p^{d-1}$-extensions $L'/K$, with $L_0\subset L'\subset L$, and that of 
$\Q_p$-subspaces $\bar L\subset \Q_p\otimes L/L_0$, of dimension $d-2$, namely, the Grassmannian of $d-2$ dimensional subspaces in a $d-1$ dimensional $\Q_p$-vector space.  
When $d>2$, the set is infinite, so there are infinitely many distinct $L'$ of such type.
Hence, $\xi_0-\eta_0$, being divisible by infinitely many relatively prime elements in $\Lambda_\Gamma$, must be $0$.

\section{The isotrivial case}
\label{s:isotrivial}

In this section, we assume that $A(L)[p^\infty]$ is infinite. Under such condition,
$A/K$ is known to be isotrivial \cite{blv09} and since $A_{p^\infty}(\bar K)=\Q_p/\Z_p$, we have $A(L)[p^\infty]=A_{p^\infty}(\bar K)$. 

Suppose $A$ is a twist of a constant ordinary elliptic curve $B$ defined over $\mathbb{F}_q$.
Let $\varphi : \operatorname{Gal}(K^{\mathrm{sep}}/K) \longrightarrow \operatorname{Aut}(B)$ be the twisting cocycle. All endomorphisms of $B$ are defined over $\F_{q}$ \cite[Theorem 2(c)]{tat66}. Since the action of $\text{Gal}(K^s/K)$ on $\text{Aut}(B)$ is trivial, the cocycle relation reduces to $\varphi(\sigma\tau) = \varphi(\sigma)\varphi(\tau)$, hence $\varphi$ is a homomorphism.
Set $F := (K^{\mathrm{sep}})^{\ker(\varphi)}$. Then $\varphi$ is decomposed as
\begin{equation}\label{e:phidecom}
\xymatrix{\varphi:\Gal(K^s/K) \ar@{->>}[r] & \Gal(F/K)\ar@{^(->}[r]^-{\bar\varphi} & \mathrm{Aut}(B).}
\end{equation}
\subsection{The Iwasawa Main Conjecture}\label{su:imciso}
Here is the key lemma whose proof is given in \S\ref{su:iso3}.
\begin{lemma}\label{isolemma3}
If $A/K$ is not a constant curve, then $p=2$ and $F$ is an unramified quadratic intermediate extension of $L/K$.
\end{lemma}

Following the lemma, to prove the Iwasawa Main Conjecture, which is known to hold for constant curves, we may assume that $F/K$ is an unramified quadratic intermediate extension of $L/K$.
Write $\Phi$ for $\Gal(L/F)$.

For a $\Lambda_\Gamma$-module $W$, define the $\varphi$-twist $\tensor[^{[\varphi]}]W{}$ to be the $\Lambda_\Gamma$-module
whose underlying $\Z_p$-module is the same as $W$, while if $P\in \tensor[^{[\varphi]}]W{}$ corresponds to
an element $Q\in W$, then for $\tau\in\Gal(K^s/K)$, $^\tau P$ is the element corresponding to
$\varphi(\tau)(^\tau Q)$.

In particular, as $\Lambda_\Phi$-modules, $\tensor[^{[\varphi]}]W{}=W$. Also, if $f\cdot W=0$, for some
$f\in\Lambda_\Gamma$, then
$$f_\varphi\cdot \tensor[^{[\varphi]}]W{}=f_{\varphi^{-1}}\cdot \tensor[^{[\varphi]}]W{}=0,$$
because $\varphi=\varphi^{-1}$.
Thus, for finitely generated $\Lambda_\Gamma$-modules $W$,
\begin{equation}\label{e:varphich}
\mathrm{CH}_{\Lambda_\Gamma}(\tensor[^{[\varphi]}]W{})=\mathrm{CH}_{\Lambda_\Gamma}(W)_\varphi.
\end{equation}
Here, for an ideal $I$ in $\Lambda_\Gamma$, $I_\varphi$ denotes the image of $I$ under the ring automorphism
$f\mapsto f_\varphi$.

The following two lemmas are proved in \S\ref{su:lemmas} and \S\ref{su:lemL}.

\begin{lemma}\label{l:xvarphi} Let $A$ be a non-constant elliptic curve. As $\Lambda_\Gamma$-modules,
\begin{equation}\label{e:phitwist}
\tensor[^{[\varphi]}]X{_{B/L}}=X_{A/L}.
\end{equation}
\end{lemma}
The lemma, the equality \eqref{e:varphich}, and the restriction formula together imply
\begin{equation}\label{e:xphi}
\mathrm{CH}_{\Lambda_\Phi}(X_{B/L})\cdot \Lambda_\Gamma=\mathrm{CH}_{\Lambda_\Gamma}(X_{B/L})^\Gamma_\Phi\cdot \Lambda_\Gamma=\mathrm{CH}_{\Lambda_\Gamma}(X_{B/L})\cdot \mathrm{CH}_{\Lambda_\Gamma}(X_{A/L}).
\end{equation}

\begin{lemma}\label{l:prod}
In $\Lambda_\Gamma$, the ideals
$$((\mathscr L_{B/L/K})_\varphi)=(\mathscr L_{A/L/K}).$$
\end{lemma}

The lemma together with the restriction formula imply
\begin{equation}\label{e:xL}
\mathscr L_{B/L/F}\cdot\Lambda_\Gamma=\mathscr L_{B/L/K}\cdot \mathscr L_{A/L/K}\cdot\Lambda_\Gamma.
\end{equation}

Since
$$\mathscr L_{B/L/F}\cdot\Lambda_\Gamma=\mathrm{CH}_{\Lambda_\Phi}(X_{B/L})\cdot \Lambda_\Gamma,$$
$$\mathscr L_{B/L/K}\cdot \Lambda_\Gamma=\mathrm{CH}_{\Lambda_\Gamma}(X_{B/L}),$$
this yields the equality \eqref{e:imc}
$$\mathscr L_{A/L/K}\cdot \Lambda_\Gamma=\mathrm{CH}_{\Lambda_\Gamma}(X_{A/L}).$$

\subsection{The proof of Lemma \ref{l:xvarphi}}\label{su:lemmas}
Fix an $L$-isomorphism of $p$-divisible groups
\[
\iota:\ A_{p^\infty}\times_K L \xrightarrow{\ \sim\ } B_{p^\infty}\times_K L
\quad\text{with}\quad
^g\iota=\iota\circ[\varphi(g)],\ \ g\in\Gamma.
\tag{N}
\]
Functoriality of flat cohomology gives

\(\iota_*:\coh^1_{\mathrm{fl}}(L,A_{p^\infty})\xrightarrow{\ \sim\ }\coh^1_{\mathrm{fl}}(L,B_{p^\infty})\).
For $g\in\Gamma$ and a morphism $f:G\to H$ of $L$-group schemes, put
$^g f := g_H^{}\circ f\circ g_G^{-1}$. Here $g_H^{}$ and $g_G^{}$ respectively 
denote the action of $g$ on $H$ and $G$. Then on cohomology we have the naturality
\[
g_H^{}\circ f_*=(^g f)_*\circ g_G^{}.
\]
Applying this to $f=\iota$ and using (N) gives
\[
g_B^{}\circ \iota_* \;=\; (^g\iota)_*\circ g_A^{}
\;=\; (\iota\circ[\varphi(g)])_*\circ g_A^{}
\;=\; \iota_*\circ[\varphi(g)]_*\circ g_A^{}. \tag{$\ast$}
\]



Since $A$ is non-constant, by Lemma \ref{isolemma3}, we have $p=2$ and $\varphi(\Gamma)\subset \Aut(B_{\bar\F_q})=\{\pm 1\}$.
Hence $[\varphi(g)]$ is either $[1]$ or $[-1]$ on $B_{p^\infty}$, and the induced map on
cohomology is multiplication by $\varphi(g)\in\{\pm 1\}$.

Since $[\pm1]$ acts by $\pm1$ on $p$-power torsion, we have $[\varphi(g)]_*=\varphi(g)$ on
cohomology. Thus \((\ast)\) says exactly that
\[
\iota_*:\ \coh^1_{\mathrm{fl}}(L,A_{p^\infty}) \xrightarrow{\ \sim\ }
\tensor[^{[\varphi]}]{\coh^1_{\mathrm{fl}}(L,B_{p^\infty})}{}
\]
is $\Gamma$-equivariant, where the twisted action on the target is
$g\cdot_\varphi y:=\varphi(g)\,(g\cdot y)$.

The same argument after base change to $L_w$ gives $\Gamma$-equivariant 
\[
\coh^1_{\mathrm{fl}}(L_w,A_{p^\infty})\xrightarrow{\ \sim\ }\tensor[^{[\varphi]}]{\coh^1_{\mathrm{fl}}(L_w,B_{p^\infty})}{}
\]
for all $w$. Functoriality of the Kummer maps shows these
isomorphisms identify the local images, hence
\[
\mathrm{Sel}_{p^\infty}(A/L)\ \xrightarrow{\ \sim\ }\ \tensor[^{[\varphi]}]{\mathrm{Sel}_{p^\infty}(B/L)}{}
\]
as discrete $\Gamma$-modules.

For any discrete $\Gamma$-module $M$ and finite-order character $\psi$ there is a
canonical identification
\[
\tensor[^{[\psi]}]M{}^\vee{} \ \cong\ \bigl(\tensor[^{[\psi^{-1}]}]M{}\bigr)^\vee,\qquad
M^\vee:=\mathrm{Hom}_{\mathrm{cts}}(M,\Q_p/\Z_p),
\]
coming from $(^g \lambda)(m)=\lambda(g^{-1}\cdot m)$, for $\lambda\in M^\vee$.
Dualizing the Selmer isomorphism and using $\varphi^{-1}=\varphi$ gives
\[
\tensor[^{[\varphi]}]X{_{B/L}}\ \cong\ X_{A/L},
\]
i.e. $\tensor[^{[\varphi]}]{X}{_{B/L}}=X_{A/L}$ as $\Lambda_\Gamma$-modules.

\subsection{The proof of Lemma \ref{l:prod}}\label{su:lemL}
Since $\dag_{B/L}\not=0$, $\dag_{A/L}\not=0$, it is sufficient to show that for some $u\in\Lambda_\Gamma^*$,
$$\dag_{B/L}\cdot \dag_{A/L} \cdot ((\mathscr L_{B/L})_\varphi-u\cdot\mathscr L_{A/L})=0,$$
or equivalently that if $\omega\in\hat\Gamma$, such that $\omega(\dag_{B/L}\cdot \dag_{A/L})\not=0$,  then
$$\omega((\mathscr L_{B/L})_\varphi)=\omega(u)\cdot\omega(\mathscr L_{A/L}),$$
since $\{\omega\in\hat\Gamma\;\mid \; \omega(\dag_{B/L}\cdot\dag_{A/L})\not=0\}\subset\hat\Gamma$ is open dense in the topology of Monsky \cite[Theorem 2.2 and 2.6]{monsky}.
For such $\omega$, we can write 
$$\omega((\mathscr L_{B/L})_\varphi)=\omega\varphi(\mathscr L_{B/L})=\star_{B,\omega\varphi}\cdot L_{B/K}(\omega\varphi,1),$$
and
$$\omega(\mathscr L_{A/L})=\star_{A,\omega}\cdot L_{A/K}(\omega,1)=\star_{A,\omega}\cdot L_{B/K}(\omega\varphi,1).$$
To compare $\star_{B,\omega\varphi}$ and $\star_{A,\omega}$, we check the fudge factors on the corresponding right hand-sides of \eqref{e:fudge} as follows. In the discussion below, the fudge factors will have the notation
$B$ or $A$ added to their index, thus, instead of $\alpha_{D_{\omega\varphi}}$ and $\alpha_{D_\omega}$, we shall see $\alpha_{B,D_{\omega\varphi}}$ and $\alpha_{A,D_\omega}$, and so on.
Since $F/K$ is everywhere unramified (Lemma \ref{l:Funram}), we have $\Delta_B=\Delta_A$ and 
$D_{\omega\varphi}=D_\omega$. If $\epsilon$ is an idele or a divisor, with image $\bar\epsilon\in\Gamma$,
we write $\varphi(\epsilon)$ for $\varphi(\bar\epsilon)$. 

By comparison, we obtain the following formulae.

 \begin{equation}\label{e:compdag}
 \omega\varphi(\dag_{B/L})=\prod_{v\in S_1} \varphi(\sigma_v) \cdot \omega(\dag_{A/L}).
 \end{equation}
 Indeed, if $\varphi(\Gamma_v)=1$, then $B_v=A_v$, so 
 $\omega\varphi(\dag_{B/L,v})=\omega(\dag_{A/L,v})$; if $v\in S_1$ and $\varphi(\sigma_v)=-1$, then $\F_{q_v^2}$ is the residue field of $F_v$, hence contained in $L_v$, so
$$\omega\varphi(\dag_{B/L,v})=\lambda_{B,v}-\omega(\sigma_v)\varphi(\sigma_v)=\varphi(\sigma_v)\cdot (\lambda_{A,v}-\omega(\sigma_v))=\varphi(\sigma_v)\cdot\omega(\dag_{A/L,v}).$$

If $\pi_v$ is a prime element of $K_v$, then $\varphi(\pi_v)\cdot \alpha_{B,v}=\alpha_{A,v}$ and
$\varphi(\pi_v)\cdot \lambda_{B,v}=\lambda_{A,v}$, so
\begin{equation}\label{e:comalpha}
\alpha_{B,D_{\omega\varphi}}=\varphi(D_{\omega})^{-1}\cdot \prod_{v\in S_m\cap\mathrm{Supp}(D_\omega)}\varphi(\pi_v)\cdot \alpha_{A,D_\omega}.\\
\end{equation}

If $a$ is an associated differential idele in \cite[\S 1.3.2]{tan24}, then
\begin{equation}\label{e:comptau}
\tau_{\omega\varphi}=\varphi(\bar a)^{-1}\cdot \varphi(D_\omega)^{-1}\cdot \tau_\omega.
\end{equation}
Finally, since for $v\not\in \mathrm{Supp}(D_\omega)$, $\varphi([v]_L)=\varphi(\pi_v)=\pm 1$, so
for $v\in S_o$,
$$(1-\alpha_{B,v}^{-1}\omega([v]_L)\varphi([v]_L))(1-\alpha_{B,v}^{-1}\omega([v]_L)^{-1}\varphi([v]_L)^{-1})=
(1-\alpha_{A,v}^{-1}\omega([v]_L))(1-\alpha_{A,v}^{-1}\omega([v]_L)^{-1}),$$

and for $v\in S_m$,
$$\lambda_{B,v}-\omega([v]_L)^{-1}\varphi([v]_L)^{-1}
=\varphi(\pi_v)\cdot (\lambda_{A,v}-\omega([v]_L)^{-1}).$$

Thus,
\begin{equation}\label{e:compXi}
\Xi_{B,S,\omega\varphi} =\prod_{\stackrel{v\in S_m}{v\not\in \mathrm{Supp}(D_\omega)}}\varphi(\pi_v)\cdot \Xi_{A,S,\omega}.
\end{equation}
The formulae \eqref{e:compdag} $\thicksim$ \eqref{e:compXi}  together imply that 
$$u=\prod_{v\in S_1}\varphi(\sigma_v)\cdot \prod_{v\in S_m} \varphi(\pi_v)\cdot \varphi(\bar a)=\pm 1\in\Lambda_\Gamma^*.$$

In the following discussion we
assume that $F\not=K$.

\subsection{The $l$-power torsion subgroups}\label{su:ppt}
For a prime number $l$ and $\nu=0,1,...,\infty$, denote $A[l^\nu] = A_{l^\nu}(\bar K)$,
$B[l^\nu]=B_{l^\nu}(\bar{\F}_q)$.
Write $\mathcal K_l:=K(A[l^\infty])$ and $\mathcal F_l:=\F_q(B[l^\infty])$. Since $A/F=B/F$,
\begin{equation}\label{e:klfl}
\mathcal K_l F=\mathcal F_l F.
\end{equation}
\begin{lemma}\label{isolemma0} We have $F\subset \K_l\mathcal F_l$.
\end{lemma}\begin{proof}

Put $\mathcal G:=\Gal(K^s/\K_l\mathcal F_l)$. Since $A[l^\infty]\subset A(\K_l)$, $B[l^\infty]\subset B(\mathcal F_l)$, for every $\tau\in \mathcal G$, $P\in A[l^\infty]$, $Q\in  B[l^\infty]$,
$$^\tau P=P,\quad ^\tau Q=Q,$$
or, by the twisting,
$$^\tau (\varphi(\tau)Q)=\varphi(\tau)(^\tau Q)=Q,\quad ^\tau Q=Q.$$
These imply that the automorphism $\varphi(\tau)$ fixes every $Q\in  B[l^\infty]$, so
the endomorphism $id-\varphi(\tau)$ has kernel containing $ B[l^\infty]$. The only possibility is that
$\varphi(\tau)=id $. Thus, $\varphi(\tau)=id$, for $\tau$ fixing the field $\K_l\mathcal F_l$.
In other words, if $\tau$ fixes $\K_l\mathcal F_l$, then it also fixes $F$.

\end{proof}

\begin{lemma}\label{l:Funram}The extension $F/K$ is unramified everywhere.
\end{lemma}
\begin{proof}
Since $A/K$ is isotrivial, its $j$-invariant is constant, hence $j(A)\in \F_q$.
In particular, for every place $v$ of $K$, writing $v(\cdot)$ for the additive
normalized valuation of $K_v$, we have $v(j(A))\ge 0$, so $A$ has no multiplicative reduction at every place of $K$. As $A/K$ is semistable, it follows that $A$ has good reduction at every place of $K$.

Choose a prime number $l\not=p$. By the N\'{e}ron--Ogg--Shafarevich criterion,
$\K_l/K$ is everywhere unramified. Since $\mathcal F_l\subset \bar\F_q$, the extension
$\K_l\mathcal F_l/K$ is also everywhere unramified, and so is the subextension
$F\subset \K_l\mathcal F_l$ over $K$.
\end{proof}

Put
$$\mathcal{M} =\F_q(B[p]).$$

\begin{lemma}\label{isolem1}
The Galois group $\Gal(\mathcal F_p/\F_q) \cong \mathbb{Z}_p \times \operatorname{Gal}(\mathcal{M}/\mathbb{F}_q)$, where $\Gal(\mathcal{M}/\mathbb{F}_q)$ can be viewed as a subgroup of $\F_p^*$. Hence $\mathcal F_p=\F_{q^\infty} \M$. If $p=2$, then $\M=\mathbb{F}_q$.
\end{lemma}
\begin{proof}
For $\nu=1$ and $\nu=\infty$, the actions of $\operatorname{Gal}\left( \mathbb{F}_q(B[p^\nu])/\mathbb{F}_q \right)$
on $B[p^\nu]$
give rise to the commutative diagram
$$\xymatrix{
\operatorname{Gal}\left( \mathcal F_p/\mathbb{F}_q \right) \ar@{^(->}[r] \ar@{->>}[d] & \mathrm{Aut}(B[p^\infty]) \ar[r]^-\sim \ar@{->>}[d] & \mathbb{Z}_p^* \ar@{->>}[d]^-{\mathsf {mod}} \\
\operatorname{Gal}\left( \M/\mathbb{F}_q \right)  \ar@{^(->}[r] & \mathrm{Aut}(B[p]) \ar[r]^-\sim & \F_p^*,} 
$$
where $\mathsf{mod}$ is the reduction modulo $p$. If $p\not=2$, then $\Z_p^times=\ker (\mathsf{mod})\times \F_p^*=\Z_p\times \F_p^*$. By the diagram, $\operatorname{Gal}(\mathcal F_p/\mathbb{F}_q) =H \times \operatorname{Gal}(\mathcal{M}/\mathbb{F}_q)$ and $H$ is a closed subgroup of $\Z_p$, and hence isomorphic to $\Z_p$. Since $\F_{q^\infty}$ is the only $\Z_p$-extension of $\F_q$, the lemma follows.

If $p=2$, then $\F_p^*$ is the trivial group, so $B[p]\subset B(\F_q)$. In this case $\Z_2^*\simeq \Z_2\times \Z/2\Z$. The diagram also shows that $\operatorname{Gal}(\mathcal F_p/\mathbb{F}_q)$ is
isomorphic to a closed subgroup $H$ of $\Z_2\times \Z/2\Z$. Since $\F_q$ has a unique quadratic extension,
$H/2H$ must be isomorphic to $\Z/2\Z$, so $H\simeq \Z_2$ (Nakayama lemma) as desired.

\end{proof}

\subsection{The proof of Lemma \ref{isolemma3}}\label{su:iso3}

\begin{lemma}\label{isolemma2}
Under the assumption that $A_{p^\infty}(L)$ is infinite, we have $\M K\subset F \subset \mathcal{M}L$.
\end{lemma}
\begin{proof}
Since $B(FL)=A(FL)$ and $A_{p^\infty}(L)=A[p^\infty]$, we have $B_{p^\infty}(FL)=B[p^\infty]$, hence
\begin{equation}\label{e:kpfl}
  K^{(p)}_\infty\mathcal{M} = K\mathcal{F}_p \subset FL.
\end{equation}
In particular, $\mathcal{M}\subset FL$.

Put $\mathcal{G}:=\Gal(FL/K)$. Since $L/K$ is pro-$p$ and $F/K$ is finite of degree prime to $p$,
the subgroup $\Gal(FL/F)\cong\Gal(L/L\cap F)$ is pro-$p$ and the quotient
$\mathcal{G}/\Gal(FL/F)\cong\Gal(F/K)$ has order prime to $p$.
In particular, $\Gal(FL/F)$ is the pro-$p$ Sylow subgroup of $\mathcal{G}$, and every prime-to-$p$
quotient of $\mathcal{G}$ factors through $\Gal(F/K)$.

Since $[K\mathcal{M}:K]$ is prime to $p$, the subextension $K\mathcal{M}/K$ of $FL/K$
corresponds to a prime-to-$p$ quotient of $\mathcal{G}$; hence
$\Gal(FL/K\mathcal{M})\supset \Gal(FL/F)$, which gives $K\mathcal{M}\subset F$.

Similarly, by \eqref{e:kpfl}, $K^{(p)}_\infty$ is a pro-$p$ subextension of $FL/K$.
Its Galois group $\Gal(K^{(p)}_\infty/K)\cong \Z_p$ is a quotient of the pro-$p$ Sylow subgroup $\Gal(FL/F)$. Since $FL/L$ is finite of degree prime to $p$, any pro-$p$ quotient of
$\Gal(FL/K)$ factors through $\Gal(L/K)$, hence $\Gal(K^{(p)}_\infty/K)$ is also a quotient of
$\Gal(L/K)$. Therefore $K^{(p)}_\infty\subset L$.

It follows that $\mathcal{M}L\supset K^{(p)}_\infty\mathcal{M}=K\mathcal{F}_p$.
By Lemma~\ref{isolemma0}, we have $F\subset \mathcal{M}L$.
\end{proof}
Because
$B/\F_q$ is ordinary, by \cite[III.10.1, IV.4.5, IV.5.7]{sil86}, the following lemma holds.

\begin{lemma}\label{l:aut}
Let $B/\F_q$ be an ordinary elliptic curve.
\begin{enumerate}
\item If $p\not=2,3$, then $\Aut(B_{\bar{\F}_q})$ is cyclic of order $2$, $4$, or $6$.
\item If $p=2$ or $p=3$, then $\Aut(B_{\bar{\F}_q})=\{\pm 1\}\cong \Z/2\Z$.
\end{enumerate}
\end{lemma}

Suppose $[F:K]$ is prime to $p$. Lemma \ref{isolemma2} says $F = \mathcal M K$, and since $\M$ is a finite field,
$F/K$ is a constant field extension. Moreover, since $\F_q$ is the constant field of $K$, $\mathcal M/\mathbb{F}_q$ and $K/\mathbb{F}_q$ are disjoint, so $\Gal(F/K)$ can be identified with $\Gal(\mathcal M/\mathbb{F}_q)$, $\sigma\mapsto\sigma_0$.
Thus, $\bar\varphi$ induces a homomorphism $\bar\varphi_0:\Gal(\M/\F_q)\longrightarrow \Aut(B_{\bar{\F}_q})$, by
$\bar\varphi_0(\sigma_0) =\bar\varphi(\sigma)$. Let $C$ be the twist of $B/\mathbb{F}_q$ via $\bar\varphi_0$.
Then $A = C \otimes_{\mathbb{F}_q} K$ is a constant curve, contradicting the assumption that $A/K$ is not constant.

Assume now that $p$ divides $[F:K]$. By Lemma \ref{l:aut}, this forces $p=2$ and
$\Aut(B_{\bar{\F}_q})=\{\pm 1\}$. In particular, $\Gal(F/K)$ is a finite $2$-group
mapping into $\{\pm 1\}$, hence $\Gal(F/K)$ has order $2$ and $F/K$ is quadratic.
By Lemma \ref{isolemma2}, we have $F\subset L$.


\appendix
\section{Density of Semistable Elliptic Curves with Trivial $\mu$-Invariant}\label{s:app}

In this section, we follow the notation of \cite{lst21}. Let $p>3$, let $k=\mathbf F_q$ with $q=p^e$, and let
$C/k$ be a smooth proper curve of genus $g_C$ with function field $K$.
All statements below are after base change to $\bar k$.

\subsection{The parameter spaces}

For $n\ge 1$, write $Y(n,C)$ for the space of triples $(L,g_2,g_3)$ with
$L\in\mathrm{Pic}^n(C)$ and
\[
(g_2,g_3)\in H^0(C,L^{\otimes 4})\times H^0(C,L^{\otimes 6}),
\qquad
\Delta:=\frac{g_2^{\,3}-g_3^{\,2}}{1728}\in H^0(C,L^{\otimes 12}),
\]
such that $\Delta\not\equiv 0$ and
\[
\min\bigl(3\,\operatorname{ord}_v(g_2),\,2\,\operatorname{ord}_v(g_3)\bigr)\;<\;12
\quad\text{for every closed point }v\text{ of }C.
\]
Let $\pi:Y(n,C)\to X(n,C)$ be the good $\mathbf G_m$-quotient.

In characteristic $p>3$, the condition
$\min(3\,\operatorname{ord}_v(g_2),\,2\,\operatorname{ord}_v(g_3))<12$
ensures that the triple $(L,g_2,g_3)$ gives rise to a globally minimal short
Weierstrass model over $C$, and the usual invariants satisfy $c_4=g_2$ and
$\Delta=(g_2^3-g_3^2)/1728$ on this parameter space.
Following \cite{lst21}, the Iwasawa invariant defines a function
$\mu:X(n,C)_{\bar k}\to\mathbf Z_{\ge0}$ and
\[
X_{\mu=0}^{(n)}\ :=\ \{x\in X(n,C)_{\bar k}:\ \mu(x)=0\}
\]
is Zariski open and dense for $n\gg0$ (\cite[Thm.\ 6.3.1]{lst21}).

\begin{theorem}\label{thm:X-open-dense}
For $n\gg0$, the subset
\[
\mathcal V_n \;:=\;
\Bigl\{\,x\in X(n,C)_{\bar k}\;:\;
\begin{array}{l}
\text{the corresponding elliptic curve over $K_{\bar k}:=K\otimes_k\bar k$ is semistable}\\[4pt]
\text{and }\mu(x)=0
\end{array}
\Bigr\}
\]
is Zariski open and dense in $X(n,C)_{\bar k}$.
\end{theorem}

\begin{proof}
\emph{Openness.}
We work over $\bar k$. Write $Y:=Y(n,C)_{\bar k}$ and replace $C$ by $C_{\bar k}:=C\times_k\bar k$.
Let $\rho:Y\to \mathrm{Pic}^n(C)_{\bar k}$ be the structural map.
Let $\mathcal P_n$ be the Poincar\'e line bundle on $C_{\bar k}\times \mathrm{Pic}^n(C)_{\bar k}$
(rigidified at a point of $C_{\bar k}$). By abuse of notation, we also denote by $\mathcal P_n$
its pullback $(\mathrm{id}_{C_{\bar k}}\times \rho)^*\mathcal P_n$ to $C_{\bar k}\times Y$.
Then there are tautological sections
\[
\mathbf g_2\in H^0(C_{\bar k}\times Y,\mathcal P_n^{\otimes 4}),\quad
\mathbf g_3\in H^0(C_{\bar k}\times Y,\mathcal P_n^{\otimes 6}),\quad
\boldsymbol\Delta:=\tfrac{\mathbf g_2^{\,3}-\mathbf g_3^{\,2}}{1728}\in H^0(C_{\bar k}\times Y,\mathcal P_n^{\otimes 12}).
\]
On $Y(n,C)$ (globally minimal short models in characteristic $p>3$), the usual Kodaira--N\'eron criterion
(or equivalently Tate's algorithm for minimal Weierstrass models) implies that a fiber over a closed point $v$
is additive if and only if $\operatorname{ord}_v(\Delta)>0$ and $\operatorname{ord}_v(c_4)>0$.
Here $c_4=\mathbf g_2$.
Since we work on $Y(n,C)$, the inequality
$\min(3\,\operatorname{ord}_v(g_2),2\,\operatorname{ord}_v(g_3))<12$ at every $v$
ensures that the associated short Weierstrass equation is minimal at each $v$.
Therefore the Kodaira--N\'eron criterion (equivalently, Tate's algorithm for minimal models)
applies fiberwise to detect additive reduction in terms of $c_4$ and $\Delta$.
Thus additive reduction occurs exactly at those pairs (v,y) for which $\boldsymbol\Delta(v,y)=0$ and $\mathbf g_2(v,y)=0$; i.e. the additive locus in $C_{\bar k}\times Y$ is $Z:=V(\boldsymbol\Delta,\mathbf g_2)$. The projection $\mathrm{pr}_2:C_{\bar k}\times Y\to Y$ is proper, hence $\mathrm{pr}_2(Z)$ is closed.
Therefore
\[
U_{\mathrm{ss}}\ :=\ Y\setminus \mathrm{pr}_2(Z)
\]
(the locus of curves semistable at every place) is Zariski open.
Since $\pi:Y\to X(n,C)_{\bar k}$ is a good $\mathbf G_m$-quotient and $U_{\mathrm{ss}}$ is $\mathbf G_m$-invariant,
\[
V_{\mathrm{ss}}\ :=\ \pi\big(U_{\mathrm{ss}}\big)\ \subset\ X(n,C)_{\bar k}
\]
is Zariski open. Intersecting with the open $X_{\mu=0}^{(n)}$ gives
\[
\mathcal V_n\ =\ V_{\mathrm{ss}}\ \cap\ X_{\mu=0}^{(n)},
\]
hence $\mathcal V_n$ is Zariski open.

\emph{Nonemptiness.}
We show that $V_{\mathrm{ss}}$ is nonempty for $n\gg 0$ by producing a point of $Y(n,C)_{\bar k}$
whose bad fibers (if any) are all multiplicative.

Fix $L\in \mathrm{Pic}^n(C_{\bar k})$ with $n$ large enough so that $L^{\otimes 4}$, $L^{\otimes 6}$, and $L^{\otimes 12}$
are globally generated. We first choose $g_2\in H^0(C_{\bar k},L^{\otimes 4})$ with reduced divisor, i.e.
$\operatorname{div}(g_2)$ has only simple zeros. Such a choice exists for $n\gg 0$:
indeed, the set of sections with a multiple zero is a proper closed subset of the affine space $H^0(C_{\bar k},L^{\otimes 4})$
(defined by the incidence condition ``$\operatorname{ord}_v(g_2)\ge 2$'' for some $v$), so its complement is nonempty open.
One way to see this is to consider the closed incidence subset
\[
\{(v,s)\in C_{\bar k}\times H^0(C_{\bar k},L^{\otimes 4}) : \operatorname{ord}_v(s)\ge 2\},
\]
defined by the vanishing of the section and its first jet at $v$; its image in
$H^0(C_{\bar k},L^{\otimes 4})$ is closed and proper, hence its complement is a nonempty open.
Let
\[
T\ :=\ \operatorname{Supp}\bigl(\operatorname{div}(g_2)\bigr).
\]

Next we choose $g_3\in H^0(C_{\bar k},L^{\otimes 6})$ satisfying the following three open conditions:
\begin{itemize}
\item $g_3(v)\neq 0$ for every $v\in T$;
\item $\operatorname{div}(g_3)$ is reduced (equivalently, $g_3$ has only simple zeros);
\item $\Delta=(g_2^3-g_3^2)/1728$ is not identically zero, and $\operatorname{div}(\Delta)$ is reduced.
\end{itemize}
Each of these conditions cuts out a Zariski open subset of $H^0(C_{\bar k},L^{\otimes 6})$:
the first is the complement of finitely many proper linear subspaces (given by the evaluation maps at points of $T$),
and the remaining ones are complements of proper closed subsets (nonreduced divisor is a closed incidence condition).
Hence for $n\gg 0$ there exists $g_3$ satisfying all three simultaneously.

For such a choice, if $v\in T$ then $g_2(v)=0$ and $g_3(v)\neq 0$, hence
\[
\Delta(v)=\frac{-g_3(v)^2}{1728}\neq 0.
\]
Equivalently,
\[
\operatorname{Supp}\bigl(\operatorname{div}(\Delta)\bigr)\cap
\operatorname{Supp}\bigl(\operatorname{div}(g_2)\bigr)\ =\ \emptyset.
\]
In particular, whenever $\Delta(v)=0$ one has $g_2(v)\neq 0$, so the fiber is multiplicative by the criterion above
(using $c_4=g_2$).

It remains to check that $(L,g_2,g_3)$ lies in $Y(n,C)_{\bar k}$, i.e. that the inequality
\[
\min\bigl(3\,\operatorname{ord}_v(g_2),\,2\,\operatorname{ord}_v(g_3)\bigr)\;<\;12
\quad\text{for every }v
\]
holds. Since $\operatorname{div}(g_2)$ and $\operatorname{div}(g_3)$ are reduced and disjoint,
for every $v$ we have $\operatorname{ord}_v(g_2)\in\{0,1\}$ and $\operatorname{ord}_v(g_3)\in\{0,1\}$, and never both equal to $1$.
Hence the minimum is at most $2$, so the inequality holds at every $v$.

Therefore the corresponding elliptic curve is semistable at every place, so the image point in $X(n,C)_{\bar k}$
lies in $V_{\mathrm{ss}}$, and $V_{\mathrm{ss}}$ is nonempty.

\emph{Density.}
For $n\gg 0$, the variety $Y(n,C)_{\bar k}$ is a nonempty open in a vector bundle over the irreducible variety
$\mathrm{Pic}^n(C)_{\bar k}$, hence $Y(n,C)_{\bar k}$ is irreducible.
The quotient map $Y(n,C)_{\bar k}\to X(n,C)_{\bar k}$ is surjective, so $X(n,C)_{\bar k}$ is irreducible as well.
Since $V_{\mathrm{ss}}$ is a nonempty open in $X(n,C)_{\bar k}$, it is dense.
By \cite[Thm.\ 6.3.1]{lst21}, $X_{\mu=0}^{(n)}$ is a dense open in $X(n,C)_{\bar k}$.
Therefore the intersection $\mathcal V_n=V_{\mathrm{ss}}\cap X_{\mu=0}^{(n)}$ is dense in $X(n,C)_{\bar k}$.
\end{proof}

\end{document}